\documentclass[a4paper, 12pt]{article}

\pdfoutput=1

\usepackage{color}
\usepackage{graphicx,epic} 
\usepackage{amsmath}
\usepackage{amsfonts}
\usepackage{amssymb}
\usepackage{amsthm}
\usepackage{epstopdf}
\usepackage{algorithmic}
\usepackage{setspace}
\usepackage[cm]{fullpage}

\newtheorem{remark}{\bf Remark}

\newtheorem{lemma}{\bf Lemma}

\newtheorem{theorem}{\bf Theorem}

\numberwithin{equation}{section}
\numberwithin{definition}{section}
\numberwithin{lemma}{section}
\numberwithin{remark}{section}
\numberwithin{corollary}{section}
\numberwithin{theorem}{section}
\numberwithin{example}{section}

\begin{document}

\title{An abstract inf-sup problem inspired by limit analysis in perfect plasticity and related applications}
\author{S. Sysala$^1$\footnote{corresponding author, email: \texttt{stanislav.sysala@ugn.cas.cz}}, J. Haslinger$^{1}$, B. D. Reddy$^{2}$, S. Repin$^{3,4}$ \\ \\
\small$^1$Institute of Geonics of the Czech Academy of Sciences, Ostrava, Czech Republic\\
\small$^2$University of Cape Town, South Africa\\
{\small$^3$V.A. Steklov Institute of Mathematics at	St. Petersburg, Russia}\\
\small$^4$University of Jyv\"{a}skyl\"a, Finland}

\maketitle

\begin{abstract}
This work is concerned with an abstract inf-sup problem generated by a bilinear Lagrangian and convex constraints. We study the conditions that guarantee no gap between the inf-sup and related sup-inf problems. The key assumption introduced in  the paper generalizes the well-known Babu\v ska-Brezzi condition. It is based on an inf-sup condition defined for convex cones in function spaces. We also apply a regularization method convenient for solving the inf-sup problem and derive a computable majorant of the critical (inf-sup) value, which can be used in a posteriori error analysis of numerical results. Results obtained for the abstract problem are applied to continuum mechanics. In particular, examples of limit load problems and similar ones arising in classical plasticity, gradient plasticity and delamination are introduced. 	
\end{abstract}

%\noindent

%\noindent
%{\bf Subject classification:} 49M15, 74C05, 74S05, 90C25 

%%%%%%%%%%%%% Section 1 %%%%%%%%%%%%%%%%%%%%%%%%%%%%%%%

\section{Introduction}
\label{sec_intro}

This paper is concerned with analysis of the abstract duality problem
\begin{equation}
\lambda^*:=\sup_{x\in P}\inf_{\substack{y\in Y\\ L(y)=1}}\ a(x,y) \stackrel{?}{=}\inf_{\substack{y\in Y\\ L(y)=1}}\sup_{x\in P}\ a(x,y)=:\zeta^*,
\label{duality_problem}
\end{equation}
where $P\subset X$ is a closed, convex set with $0_X\in P$, $X,Y$ are Banach spaces, $L$ is a non-trivial continuous linear functional in $Y$, and $a\colon X\times Y\rightarrow \mathbb R$ is a bilinear form continuous with respect to both arguments. Henceforth the problem in the right hand side of \eqref{duality_problem} is called primal, while the one in the left hand side is called dual. It is easy to check that $0\leq\lambda^*\leq\zeta^*\leq+\infty$. In general, necessary and sufficient conditions for $\lambda^*=\zeta^*$ are unknown (therefore, (\ref{duality_problem}) uses the symbol $\stackrel{?}{=}$). One of our main goals is to identify cases where (\ref{duality_problem}) holds as the equality.

Problem (\ref{duality_problem}) and similar problems appear in various applications, from mechanics to economics  \cite{ET74,Ch80,BBF13}. In finite dimensions, minimax and maximin variants of these problems are known in game theory \cite{Myerson_2013} and linear, cone or convex programming \cite{Dantzig_1998, Boyd_2004, Nocedal_2006, Kanno_2011}. 
	
In classical elastic-perfect plasticity, (\ref{duality_problem}) is known as the \textit{limit analysis problem}. In this case, $\lambda^*$ is the factor that determines the critical load $\lambda^*L$ ($L$ is a linear functional associated with external loads), subject to the constraint set $P$ of plastically admissible stresses; see for example \cite{J76, Ch80, T85, CL90, Ch96, RS95, Sl13, RSH18, HRS19}. For the load $\lambda L$ with $\lambda>\zeta^*$, no solution of the primal and dual problems exists; the body is unable to sustain the loading and collapses. Also, we note the similarity between (\ref{duality_problem}) and the \textit{shakedown analysis problem} (see \cite{Zouain_2018} and the references therein). 

Although the limit analysis problem has been studied for several decades, it is still unsolved in the general setting and presents  a challenging problem from the theoretical and numerical points of view. There are several reasons that stimulate further analysis of the problem. First, we notice that the equality $\lambda^*=\zeta^*$ can be analyzed in a rather general framework introduced in \cite{ET74} or by using particular results from \cite{Ch80, T85, HRS19}. However, these results do not cover any interesting cases. Second, additional and hidden constraints appear in the primal and dual problems (that follow from their inf- and sup-definitions). They often make the numerical analysis difficult. The third reason is related to the choice of the function spaces $X$ and $Y$. This question becomes especially important if the primal problem is related to minimization of a functional with linear growth at infinity and a certain problem relaxation must be done to find a minimizer (see e.g. \cite{J76,T85,RS95,Ch96}). Then we arrive, for example, at a formulation in which the $BD$- or $BV$- spaces of functions of bounded deformation and bounded variation, respectively, are appropriate for the problem setting \cite{T85}. Nevertheless, standard Sobolev spaces seem to be sufficient or even more appropriate for analysis of numerical errors \cite{Repin2010, RSH18, HRS19}. Finally, reliable estimates of $\lambda^*$ and $\zeta^*$ are often required because they define safety factors of structures. Lower bounds of $\lambda^*$ and upper bounds of $\zeta^*$ can be found by analytical approaches for specific geometries \cite{CL90} or, more generally, by finite element methods; see \cite{Sl13} and the references therein. Computable majorants of $\zeta^*$ can be found in recent papers \cite{RSH18, HRS19}.

In order to investigate the abstract problem (\ref{duality_problem}), we use the ideas applied in \cite{HRS15, HRS16, RSH18, HRS19} for analysis of limit load problems. This extension is not always straightforward and requires innovative techniques. In particular, we derive conditions for the equality $\lambda^*=\zeta^*$ to hold, the existence of a solution to the dual problem in (\ref{duality_problem}), a regularization method for solving (\ref{duality_problem}) with related convergence results, and a computable majorant of $\zeta^*$, which can be used for a posteriori analysis
of numerical results. 

One of the key assumptions in the results presented is the so-called \textit{inf-sup condition on convex cones} which was introduced in \cite{HRS19}. This condition generalizes the Babu\v ska-Brezzi condition defined on function spaces \cite{Babuska_1971, Brezzi_1974}. Conditions of this type are important for analysis of saddle point problems generated by various mixed finite element approximations \cite{BBF13}.

Generalization and abstraction of results is a basic procedure that allows results and insights in a particular application to be applied to broad classes of problems. In our case, we show that the results presented here are useful in problems of gradient-enhanced plasticity and in delamination problems. We choose the strain gradient model studied in \cite{Reddy_etal2008, Reddy2011a, CEMRS17, RS20} and use (\ref{duality_problem}) for the description of a global yield surface and for limit load analysis. In related work, limit analysis has been considered for a model in which size-dependence is through the gradient of a scalar function of plastic strain, see \cite[Section 7]{Fleck-Willis2009} or \cite{Polizzotto2010}. One can expect further applications of the problem (\ref{duality_problem}), at least within nonlinear mechanics.

The rest of the paper is organized as follows. In Section \ref{sec_analysis}, we introduce the primal and dual problems, discuss them in more detail, and present criteria ensuring their solvability and the principal duality relation $\lambda^*=\zeta^*$. One of the criteria is based on the inf-sup condition on convex cones. The proof of this new result is carried out in Section \ref{sec_proof} and its extensions are studied in Section \ref{sec_extension}. Section \ref{sec_regularization} is devoted to a regularization of the problem (\ref{duality_problem}). The regularized problem provides a lower and sufficiently sharp bound of $\lambda^*$, reduces the constraints in the dual problem, and thus it is convenient for numerical solution. In Section \ref{sec_majorant}, a computable majorant of the quantity $\zeta^*$ is derived. Section \ref{sec_examples} contains particular examples of the abstract problem (\ref{duality_problem}), including classical and strain-gradient plasticity and a delamination problem.

%%%%%%%%%%%%%%%%%%%%%%%%%%%%%%%%%%%%%%%%%%%%%%

\section{The primal and dual problems and duality criteria}
\label{sec_analysis}

First, we recapitulate the basic assumptions used in the problem (\ref{duality_problem}):
\begin{itemize}
	\item[(A1)] $X,Y$ are two Banach spaces equipped with the norms $\|.\|_X$ and $\|.\|_Y$, respectively. The corresponding dual spaces are denoted by $X^*$ and $Y^*$;
	\item[(A2)] $a\colon X\times Y\rightarrow \mathbb R$ is a continuous bilinear form; 
	\item[(A3)] $L\colon Y\rightarrow \mathbb R$ is a non-trivial continuous linear functional (i.e., $L\neq0$ in $Y^*$);
	\item[(A4)] $P\subset X$ is a nonempty, closed and convex set with $0_X\in P$.
\end{itemize}

The primal problem in (\ref{duality_problem}) reads
\begin{equation}
\zeta^*=\inf_{\substack{y\in Y\\ L(y)=1}}\sup_{x\in P}\ a(x,y)=\inf_{\substack{y\in Y\\ L(y)=1}} \mathcal J(y),
\label{dual_problem}
\end{equation}
where 
\begin{equation}
\mathcal J\colon Y\rightarrow \mathbb R\cup\{+\infty\},\qquad \mathcal J(y):=\sup_{x\in P}\ a(x,y),\qquad y\in Y.
\label{J0}
\end{equation}
The functional $\mathcal J$ is convex, proper and 1-positively homogeneous. In addition, the effective domain $\mathrm{dom}\,\mathcal J$ is a convex cone; see Section \ref{sec_majorant} for more details. We shall assume that all cones considered in the text have a vertex at zero, so henceforth do not emphasize this property. We say that the problem \eqref{dual_problem} has a solution if the functional $\mathcal J$ has a minimizer in the feasible set $\mathrm{dom}\,\mathcal J\cap\{y\in Y\ |\; L(y)=1\}$. Using the positive homogeneity of $\mathcal J$, we obtain the following useful and equivalent definition of $\zeta^*$:
\begin{equation}
\zeta^*=\sup\{\lambda\in\mathbb R_+\ |\;\;\mathcal J(y)-\lambda L(y)\geq0\;\;\forall y\in Y\}.
\label{dual_problem1}
\end{equation}

To rewrite the dual problem in (\ref{duality_problem}) we define the functional
\begin{equation}
\mathcal I(x):=\inf_{\substack{y\in Y\\ L(y)=1}}\ a(x,y)=\left\{\begin{array}{cc}
\lambda, & \exists\lambda\in\mathbb R:\;\; a(x,y)=\lambda L(y)\;\;\forall y\in Y,\\
-\infty, & \mbox{otherwise},
\end{array}
\right.\quad x\in X,
\label{I}
\end{equation}
and the related set
\begin{equation}
\Lambda_\lambda:=\{x\in X\ |\;\; a(x,y)=\lambda L(y)\;\;\forall y\in Y\}.
\label{Lambda}
\end{equation}
Then, we have
\begin{equation}
\lambda^*=\sup_{x\in P}\inf_{\substack{y\in Y\\ L(y)=1}}\ a(x,y)=\sup_{x\in P}\, \mathcal I(x)=\sup\{\lambda\in\mathbb R_+\ |\;\; P\cap \Lambda_\lambda\neq\emptyset\}.
\label{primal_problem}
\end{equation}
We shall say that the problem \eqref{primal_problem} has a solution if $\lambda^*<+\infty$ and there exists $\bar x\in P\cap\Lambda_{\lambda^*}$.

Now, we present three different results ensuring the equality $\lambda^*=\zeta^*$ and the existence of primal or dual solutions. The first result follows from \cite[Proposition VI.2.3 and Remark VI.2.3]{ET74}.

\begin{theorem}
Let (A1)--(A4) be satisfied and assume in addition that 
\begin{itemize}
	\item[(B)] $P$ is a bounded set in $X$.
\end{itemize}
Then $\lambda^*=\zeta^*$ and the dual problem (\ref{primal_problem}) has a solution.
\label{theorem_duality1}
\end{theorem}

Unfortunately, the set $P$ can be unbounded in plasticity and other applications. Therefore, we also need other criteria. The second result has been introduced in \cite[Theorem 2.1]{Ch80} and also used in \cite[Theorem 5.7]{Ch96}. It is convenient for use with non-reflexive spaces such as $L^\infty$.
\begin{theorem}
Let (A1)--(A4) be satisfied together with the following:
\begin{itemize}
	\item[$(C1)$] $P$ has a non-empty interior in $X$;
	\item[$(C2)$] There exists $x_0\in X$ such that $a(x_0,y)=L(y)$ for any $y\in Y$;
	\item[$(C3)$] For any $M\in X^*$ such that
	$$\left\{
	\begin{array}{l}
	\inf_{x\in P}\limits M(x)>-\infty,\\[3mm]
	a(x,y)=0\;\;\forall y\in Y\;\;\Longrightarrow M(x)=0,
	\end{array}
	\right.$$
    there exists $y_0\in Y$ satisfying $a(x,y_0)=M(x)$ for any $x\in X$.
\end{itemize}
Then $\lambda^*=\zeta^*$ and the primal problem has a solution. 
\label{theorem_duality2}
\end{theorem}

The third result is inspired by \cite[Theorem 5.2]{HRS19}. It is convenient for analysis on reflexive Banach spaces. This result is new and will be proven in the next section. 
\begin{theorem}
	Let (A1)--(A4) be satisfied and, in addition, assume the following:
	\begin{itemize}
		\item[$(D1)$] $X$ is a reflexive Banach space;
		\item[$(D2)$] $Y$ is a Hilbert space with a scalar product $(.,.)_Y$ and the induced norm $\|.\|_Y$;
		\item[$(D3)$] For any $x\in P$ there exist $x_A\in P_A$ and $x_C\in P_C$ such that $x=x_A+x_C$, where $P_A\subset X$ is closed, convex and bounded and $P_C\subset P$ is a closed convex cone; 
		\item[$(D4)$] 
		\begin{equation}
		\inf_{\substack{x_C\in P_{C}\\ x_C\neq0_X}}\ \sup_{\substack{y\in Y\\ y\neq0_Y}}\ \frac{a(x_C,y)}{\|x_C\|_X\|y\|_Y} =c_*>0.
		\label{inf-sup_abstract}
		\end{equation}
	\end{itemize}
	Then $\lambda^*=\zeta^*$. Moreover, if $\lambda^*<+\infty$ then the dual problem (\ref{primal_problem}) has a solution.
	\label{theorem_duality3}
\end{theorem}

It is worth noting that for the validity of the theorem it suffices to assume that the set $P_C$ is only closed and convex in $X$ and satisfies $(D4)$. On the other hand, we have
$$\frac{a(x_C,y)}{\|x_C\|_X\|y\|_Y}=\frac{a(\alpha x_C,y)}{\|\alpha x_C\|_X\|y\|_Y}\quad\forall \alpha>0.$$
This fact (independence of the scaling parameter) explains why we assume that $P_C$ is a convex cone. In addition, we shall see in Section \ref{sec_majorant} that the cones $P_C$ and $\mathrm{dom}\,\mathcal J$ are closely related.

We also note that any closed linear subspace of $X$ is a special case of the cone $P_C$. Then, we arrive at the standard inf-sup condition on function spaces. This case will be considered in Theorem \ref{theorem_duality4} and in Section \ref{sec_examples}.

%%%%%%%%%%%%%%%%%%%%%%%%%%%%%%%%%%%%%%%%%%%%%%%

\section{The proof of Theorem \ref{theorem_duality3}}
\label{sec_proof}

Within this section we assume that the conditions (A1)--(A4), (D1)--(D4) are satisfied and also $\lambda^*<+\infty$ (notice that Theorem \ref{theorem_duality3} holds trivially for $\lambda^*=+\infty$). To prove this theorem we define auxiliary functions
$\varphi\colon\mathbb R\rightarrow\mathbb R_+$ and $\Phi_\lambda\colon X\rightarrow \mathbb R_+$:
\begin{equation}
\varphi(\lambda):=\inf_{x\in P}\Phi_\lambda(x),\quad \Phi_\lambda(x):=\sup_{\substack{y\in Y\\ y\neq 0_Y}}\frac{a(x,y)-\lambda L(y)}{\|y\|_Y}.
\label{phi}
\end{equation}
Their basic properties are introduced in the following lemma.
\begin{lemma}
	The function $\Phi_\lambda$ is nonnegative, convex and Lipschitz continuous in $X$ for any $\lambda\in\mathbb R_+$. The function $\varphi$ is nonnegative, nondecreasing, and Lipschitz continuous in $\mathbb R_+$.
	\label{lem_Phi_lambda01}
\end{lemma}
\begin{proof}
	It is straightforward to verify that $\Phi_\lambda$ and $\varphi$ are nonnegative and convex. Let $x_1,x_2\in X$. Then, using continuity of the bilinear form $a$, we have
	$$\Phi_\lambda(x_1)=\sup_{\substack{y\in Y\\ y\neq 0_Y}}\frac{a(x_2,y)-\lambda L(y)+a(x_1-x_2,y)}{\|y\|_Y}\leq \Phi_\lambda(x_2)+\|a\|\|x_1-x_2\|_X,$$
	where $\|a\|$ is the norm of $a$.
	Similarly, $\Phi_\lambda(x_2)\leq \Phi_\lambda(x_1)+\|a\|\|x_1-x_2\|_X$ and so $|\Phi_\lambda(x_1)-\Phi_\lambda(x_2)|\leq \|a\|\|x_1-x_2\|_X$ proving the Lipschitz continuity of $\Phi_\lambda$ in $X$.
	
	Since $P$ is convex and $0_X\in P$, we have $x/\alpha\in P$ for any $x\in P$ and $\alpha\geq1$. Hence,
	$$\varphi(\alpha\lambda)=\alpha\inf_{x\in P}\Phi_\lambda(x/\alpha)\geq\alpha\inf_{x\in P}\Phi_\lambda(x)=\alpha\varphi(\lambda)\geq \varphi(\lambda)\quad\forall\alpha\geq1,$$
	i.e, $\varphi$ is nondecreasing in $\mathbb R_+$. Let $\lambda,\bar\lambda\in\mathbb R_+$, $\lambda<\bar\lambda$. Then $\varphi(\lambda)\leq\varphi(\bar\lambda)$ and
	$$\varphi(\bar\lambda)=\inf_{x\in P}\sup_{\substack{y\in Y\\ y\neq 0}}\frac{a(x,y)-\lambda L(y)-(\bar\lambda-\lambda)L(y)}{\|y\|_Y}\leq \varphi(\lambda)+(\bar\lambda-\lambda)\|L\|_{Y*}.$$
	Thus $\varphi$ is Lipschitz continuous in $\mathbb R_+$ with modulus $\|L\|_{Y^*}$.
\end{proof}

The next lemma shows that the function $\varphi$ is closely related to the problems \eqref{primal_problem} and \eqref{dual_problem}.

\begin{lemma}
	The function $\varphi$ defined in \eqref{phi} satisfies the following relations:
	\begin{equation}
	\varphi(\lambda)=0 \;\;\mbox{if}\;\; \lambda\leq \zeta^*, \quad \varphi(\lambda)>0 \;\;\mbox{if}\;\; \lambda> \zeta^*,
	\label{phi01}
	\end{equation}
and
	\begin{equation}
	\lambda^*=\zeta^*\quad\mbox{if and only if} \quad\varphi(\lambda)>0 \quad\forall \lambda>\lambda^*.
	\label{phi02}
	\end{equation}
	\label{lem_phi01}
\end{lemma}

\begin{proof}	
	To prove (\ref{phi01}) we use the Lagrangian
	\begin{equation}
	\mathcal L(x,y):=\frac{1}{2}\|y\|_Y^2+a(x,y)-\lambda L(y),\quad x\in P,\; y\in Y.
	\label{Lagr_L}
	\end{equation}
	The mapping $y\mapsto\mathcal L(x,y)$ is coercive, convex, and continuous in $Y$ for any $x\in P$ while $x\mapsto \mathcal L(x,y)$ is linear for any $y\in Y$ and the set $P$ is closed and convex in $X$. Therefore, by \cite[Proposition VI 2.3]{ET74}, we know that
	\begin{equation}
	\min_{y\in Y}\sup_{x\in P}\mathcal L(x,y)=\sup_{x\in P}\inf_{y\in Y}\mathcal L(x,y).
	\label{duality_L}
	\end{equation}
	For any given $x\in P$, there exists a unique element $y_x\in Y$ such that 
	$$\mathcal L(x,y_x)\leq \mathcal L(x,y)\quad\forall y\in Y,$$
	or equivalently
	\begin{equation}
	(y_x,y)_Y=\lambda L(y)-a(x,y)\quad \forall y\in Y.
	\label{v_tau}
	\end{equation}
	Consequently,
	\begin{equation}
	\|y_x\|_Y=\sup_{\substack{y\in Y\\ y\neq 0_Y}}\frac{(y_x,y)_Y}{\|y\|_Y}=\sup_{\substack{y\in Y\\ y\neq 0_Y}}\frac{-(y_x,y)_Y}{\|y\|_Y}=\sup_{\substack{y\in Y\\ y\neq 0_Y}}\frac{a(x,y)-\lambda L(y)}{\|y\|_Y}=\Phi_\lambda(x)
	\label{norm_v_tau}
	\end{equation}
	and
	\begin{equation}
	\sup_{x\in P}\inf_{y\in Y}\mathcal L(x,y)=\sup_{x\in P}\Big\{-\frac{1}{2}\|y_x\|^2_Y\Big\}\stackrel{(\ref{norm_v_tau})}{=}-\frac{1}{2}\inf_{x\in P}\Phi_\lambda^2(x)=-\frac{1}{2}\left(\inf_{x\in P}\Phi_\lambda(x)\right)^2=-\frac{1}{2}\varphi^2(\lambda).
	\label{si=is}
	\end{equation}
		
	From (\ref{duality_L}) and (\ref{si=is}), we have:
	$$
	-\frac{1}{2}\varphi^2(\lambda)=\min_{y\in Y}\sup_{x\in P}\mathcal L(x,y)=\min_{y\in Y}\left\{\frac{1}{2}\|y\|_Y^2+\mathcal J(y)-\lambda L(y)\right\}\quad\forall\lambda\in\mathbb R_+,
	$$
	where $\mathcal J$ is the primal functional defined by (\ref{J0}). Thus,
	\begin{equation}
	\varphi(\lambda)=\left(-2\min_{y\in Y}\left\{\frac{1}{2}\|y\|_Y^2+\mathcal J(y)-\lambda L(y)\right\}\right)^{1/2}.
	\label{phi2}
	\end{equation}
	From (\ref{dual_problem1}), one can see that $\mathcal J(y)-\lambda L(y)\geq0$ for any $\lambda<\zeta^*$ and $y\in Y$. Hence, $\varphi(\lambda)=0$ for any $\lambda<\zeta^*$ and $\varphi(\zeta^*)=0$ using the continuity argument. On the other hand, if $\lambda>\zeta^*$ then there exists $\bar y\in Y$ such that $\mathcal J(\bar y)-\lambda L(\bar y)<0$. Hence,
	$$\frac{1}{2}\|\alpha\bar y\|_Y^2+\mathcal J(\alpha\bar y)-\lambda L(\alpha\bar y)=\alpha\left\{\frac{\alpha}{2}\|\bar y\|_Y^2+\mathcal J(\bar y)-\lambda L(\bar y)\right\}<0$$
	for any $\alpha>0$ small enough. From this and (\ref{phi2}), it follows that $\varphi(\lambda)>0$ for any $\lambda>\zeta^*$. Therefore, (\ref{phi01}) holds. It is easy to see that (\ref{phi02}) follows from (\ref{phi01}) and the inequality $\lambda^*\leq\zeta^*$.
\end{proof}

Next, consider the following problem: {\it given} $\lambda\geq0$,
\begin{equation}
\mbox{{\it find }}x_\lambda\in P:\quad \Phi_\lambda(x_\lambda)\leq\Phi_\lambda(x)\quad\forall x\in P.
\label{problem_Phi}
\end{equation}

\begin{lemma}
	Let (\ref{problem_Phi}) have a solution for any $\lambda\geq0$. Then $\lambda^*=\zeta^*$. In addition, $P\cap \Lambda_{\lambda^*}\neq\emptyset$; that is, the dual problem \eqref{primal_problem} has a solution.
	\label{lem_equivalence}
\end{lemma}

\begin{proof}
	Let $\lambda>\lambda^*$ be fixed but arbitrary and $x_\lambda\in P$ be the solution to (\ref{problem_Phi}). From \eqref{primal_problem} and the choice of $\lambda$, it follows that $x_\lambda\not\in\Lambda_\lambda$. Using the definition \eqref{Lambda} of $\Lambda_\lambda$, we see that there exists $\bar y\in Y$ such that $a(x_\lambda,\bar y)-\lambda L(\bar y)>0$. Hence, $\varphi(\lambda)=\Phi_\lambda(x_\lambda)>0$. By Lemma \ref{lem_phi01}, we have $\lambda^*=\zeta^*$. If $\lambda\leq\lambda^*$ then $\varphi(\lambda)=\Phi_\lambda(x_\lambda)=0$ and so $x_\lambda\in P\cap \Lambda_\lambda$, proving the existence of a solution to \eqref{primal_problem}.
\end{proof}

\medskip\noindent
{\bf Proof of Theorem \ref{theorem_duality3}.}
The proof is based on Lemma \ref{lem_equivalence}. We show that the problem (\ref{problem_Phi}) has a solution for any $\lambda\geq0$ under the assumptions of Theorem \ref{theorem_duality3}. From Lemma \ref{lem_Phi_lambda01}, we know that the function $\Phi_\lambda$ is convex and Lipschitz continuous in $X$ for any $\lambda\in\mathbb R_+$. Using the assumptions (D3) and (D4) we prove that $\Phi_\lambda$ is also coercive in $P$ for any $\lambda\geq0$. Indeed, for any $x\in P$ and $\lambda\in\mathbb R_+$, we have:
	\begin{eqnarray}
	\Phi_\lambda(x)&=&\sup_{\substack{y\in Y\\ y\neq 0_Y}}\frac{a(x_C,y)+a(x_A,y)-\lambda L(y)}{\|y\|_Y}\nonumber\\
	&\geq& c_*\|x_\mathcal C\|_X-\|a\|\|x_A\|_X-\lambda\|L\|_{Y^*}\nonumber\\
	&\geq& c_*\|x\|_X-(c_*+\|a\|)\|x_A\|_X-\lambda\|L\|_{Y^*}\nonumber\\
	&\geq& c_*\|x\|_X-(c_*+\|a\|)\rho_A-\lambda\|L\|_{Y^*}\quad\forall x\in P,
	\label{coercivity}
	\end{eqnarray}
	where $c_*>0$ is the inf-sup constant from \eqref{inf-sup_abstract} and $\rho_A$ is a positive constant characterizing the boundedness of $A$. Since $X$ is a reflexive Banach space, the properties of $\Phi_\lambda$ guarantee that (\ref{problem_Phi}) has a solution for any $\lambda\geq0$. %Therefore, Theorem \ref{theorem_duality3} holds.
\qed

%%%%%

\section{Generalizations of Theorem \ref{theorem_duality3}}
\label{sec_extension}

Now we present three different generalizations (or extensions) of Theorem \ref{theorem_duality3}. Since their proofs are quite analogous, we only sketch them.

First, the assumption (D4) cannot hold if the subspace
\begin{equation}
H:=\{x_0\in X\ |\;\; a(x_0,y)=0\quad\forall y\in Y\}
\label{X0}
\end{equation}
contains an element $x_0\in P_C$ such that $x_0\neq0$. In Section \ref{sec_examples}, we will show that this case may arise in some applications. To weaken the assumption (D4) we introduce the quotient space $X/ H$ (with the norm $\|\cdot\|_{X/ H}$) whose elements are the equivalence classes induced by the equivalence relation:
$$x_1\cong x_2\quad\mbox{if and only if}\quad x_1-x_2\in H,\quad x_1,x_2\in X.$$
Let $P/H$ denote the set of equivalent classes generated by the set $P$. It is easy to verify that $P/H$ is a closed, convex, and nonempty set in $X/ H$. Similarly, one can introduce the sets $P_A/H$ and $P_C/H$, where $P_A$ and $P_C$ are defined in accordance with the assumption (D3). These sets have properties analogous to properties of $P_A$ and $P_C$. In particular, for any $x\in P/H$ there exists $x_A\in P_A/H$ and $x_C\in P_C/H$ such that $x=x_A+x_C$. We note that if $X$ is a Hilbert space then $X/ H$ can be identified with the orthogonal complement $H^\perp$  of $H$ in $X$ and $P/H$ with the projection of $P$ onto $H^\perp$. 

\begin{theorem}
	Let the assumptions (A1)--(A4) and (D1)-(D3) of Theorem \ref{theorem_duality3} be satisfied, $H$ be defined by \eqref{X0}, and
		\begin{equation}
		\inf_{\substack{x_C\in P_{C}/H\\ x_C\neq0_X}}\ \sup_{\substack{y\in Y\\ y\neq 0_Y}}\ \frac{a(x_C,y)}{\|x_C\|_{X/ H}\|y\|_Y} =c_*>0.
		\label{inf-sup_abstract4}
		\end{equation}
	Then $\lambda^*=\zeta^*$. If, in addition $\lambda^*<+\infty$ then the dual problem (\ref{primal_problem}) has a solution.
	\label{theorem_duality6}
\end{theorem}

\noindent
{\it Sketch of the proof:} It suffices to show that \eqref{problem_Phi} has a solution in $P$ for any $\lambda\geq0$ under the assumptions of this theorem. Let $\lambda\geq 0$ be fixed. Using \eqref{X0}, we see that the function $\Phi_\lambda$ defined in \eqref{phi} satisfies
\begin{equation}
\Phi_\lambda(x+x_0)=\Phi_\lambda(x)\quad\forall x\in X,\;\forall x_0\in H.
\label{Phi_prop}
\end{equation}
Therefore, \eqref{problem_Phi} has a solution in $P$ if and only if $\Phi_\lambda$ has a minimum in $P/H$. From \eqref{inf-sup_abstract4}, one can prove the coercivity of $\Phi_\lambda$ in $P/H$ analogously as in the proof of Theorem \ref{theorem_duality3}. Therefore, $\Phi_\lambda$ has a minimum in $P/H$ and thus the result of Theorem \ref{theorem_duality6} holds. \qed

Second, it turns out that the assumption (D2) of Theorem \ref{theorem_duality3} can be extended to some reflexive Banach spaces associated with a bounded Lipschitz domain $\Omega\subset \mathbb R^d$, $d=2,3$.

\begin{theorem}
	Let the assumptions (A1)--(A4) and (D1), (D3)-(D4) be satisfied and 
	\begin{itemize}
		\item[(D2')] $Y=W^{1,p}(\Omega,\mathbb R^m)$, equipped with the standard Sobolev norm 
		$$\|y\|_Y=\left(\int_\Omega|\nabla y|^p+|y|^p\,dx\right)^{1/p}.$$
	\end{itemize}
	Then $\lambda^*=\zeta^*$. In addition, if $\lambda^*<+\infty$ then the dual problem (\ref{primal_problem}) has a solution.
	\label{theorem_duality4}
\end{theorem}

\noindent
{\it Sketch of the proof:} It suffices to modify formulae (\ref{Lagr_L})--(\ref{phi2}). To this end, we set 
\begin{equation}
\mathcal L(x,y):=\frac{1}{p}\|y\|_Y^p+a(x,y)-\lambda L(y),\quad x\in P,\; y\in Y.
\label{Lagr_L2}
\end{equation}
Then (\ref{duality_L}) holds and there exists a unique element $y_x\in Y$ such that 
$$\mathcal L(x,y_x)\leq \mathcal L(x,y)\quad\forall y\in Y,$$
or equivalently
\begin{equation}
\int_\Omega|\nabla y_x|^{p-2}\nabla y_x:\nabla y+|y_x|^{p-2}y_x\cdot y\,dx=\lambda L(y)-a(x,y)\quad \forall y\in Y.
\label{v_tau2}
\end{equation}
Consequently,
\begin{align}
\|y_x\|_Y^{p-1}=\sup_{\substack{y\in Y\\ y\neq 0_Y}}\frac{\int_\Omega|\nabla y_x|^{p-2}\nabla y_x:\nabla y+|y_x|^{p-2}y_x\cdot y\,dx}{\|y\|_Y}=\sup_{\substack{y\in Y\\ y\neq 0_Y}}\frac{a(x,y)-\lambda L(y)}{\|y\|_Y}=\Phi_\lambda(x)
\label{norm_v_tau2}
\end{align}
and
\begin{equation}
\sup_{x\in P}\inf_{y\in Y}\mathcal L(x,y)=\sup_{x\in P}\Big\{-\frac{1}{q}\|y_x\|^p_Y\Big\}\stackrel{(\ref{norm_v_tau2})}{=}-\frac{1}{q}\inf_{x\in P}\Phi_\lambda^q(x)=-\frac{1}{q}\left(\inf_{x\in P}\Phi_\lambda(x)\right)^q=-\frac{1}{q}\varphi^q(\lambda),
\label{si=is2}
\end{equation}
where $1/q=1-1/p$. The rest of the proof is analogous to that of Section \ref{sec_proof}. \qed

The third extension illustrates that Theorem \ref{theorem_duality3} remains valid even if the space $Y$ is replaced by its conic subset.

\begin{theorem}
	Let the assumptions (A1)--(A4) and (D1)--(D3) of Theorem \ref{theorem_duality3} be satisfied, $Y_C\subset Y$ be a closed convex cone, and
		\begin{equation}
		\inf_{\substack{x_C\in P_{C}\\ x_C\neq0_X}}\ \sup_{\substack{y\in Y_C\\ y\neq 0_Y}}\ \frac{-a(x_C,y)}{\|x_C\|_X\|y\|_Y} =c_*>0.
		\label{inf-sup_abstract2}
		\end{equation}
Then
\begin{equation}
\lambda^*:=\sup_{x\in P}\inf_{\substack{y\in Y_C\\ L(y)=1}}\ a(x,y) =\inf_{\substack{y\in Y_C\\ L(y)=1}}\sup_{x\in P}\ a(x,y)=:\zeta^*.
\label{duality_problem2}
\end{equation}
In addition, $\lambda^*=\max\{\lambda\in\mathbb R_+\ |\;\; P\cap\Lambda_{\lambda}\neq\emptyset\}$, where $$\Lambda_{\lambda}=\{x\in X\ |\; a(x,y)\geq\lambda L(y)\;\forall y\in Y_C\}.$$
\label{theorem_contact}
\end{theorem}

\noindent
{\it Sketch of the proof:} The following two changes in the proof of Theorem \ref{theorem_duality3} are necessary: 
\begin{enumerate}
	\item We define
	\begin{equation*}
	\Phi_\lambda(x):=\sup_{\substack{y\in Y_C\\ y\neq 0_Y}}\frac{-a(x,y)+\lambda L(y)}{\|y\|_Y}
	\label{Phi}
	\end{equation*}
	and consider the minimization problem
	$$\mbox{find } y_x\in Y_C:\quad \mathcal L(x,y_x)\leq \mathcal L(x,y)\quad\forall y\in Y_C$$
	with $\mathcal L$ defined by \eqref{Lagr_L}. Since $Y_C$ is a closed convex cone, the corresponding necessary and sufficient condition characterising $y_x$ reads 
	\begin{equation*}
	(y_x,y)_Y\geq \lambda L(y)-a(x,y)\quad \forall y\in Y_C,\quad (y_x,y_x)_Y=\lambda L(y_x)-a(x,y_x).
	\label{v_tau3}
	\end{equation*}
	We obtain
	\begin{equation*}
	\|y_x\|_Y=\sup_{\substack{y\in Y_C\\ y\neq 0_Y}}\frac{(y_x,y)_Y}{\|y\|_Y}\geq\sup_{\substack{y\in Y_C\\ y\neq 0_Y}}\frac{-a(x,y)+\lambda L(y)}{\|y\|_Y}\geq\frac{-a(x,y_x)+\lambda L(y_x)}{\|y_x\|_Y}=\|y_x\|_Y,
	\label{norm_v_tau3}
	\end{equation*}
	so that $\|y_x\|_Y=\Phi_\lambda(x)$.
	\item To prove Lemma \ref{lem_equivalence}, we modify \eqref{primal_problem} as follows:
	$$\lambda^*=\sup\{\lambda\in\mathbb R_+\ |\;\; P\cap \Lambda_\lambda\neq\emptyset\},\quad \Lambda_\lambda:=\{x\in X\ |\;\; a(x,y)\geq\lambda L(y)\;\;\forall y\in Y_C\}.$$
\end{enumerate}
Then the proof of Theorem \ref{theorem_duality3} is applicable without any substantial changes.

%%%%%%%%%%%%%%%%%%%%%%%%

\section{Regularization method}
\label{sec_regularization}

Regularization methods are often used for solving nonsmooth, constrained, or ill-posed problems. As an example, we mention proximal point methods \cite{Parikh_Boyd_2014} which can be used for solving the problems \eqref{primal_problem} and \eqref{dual_problem}.

Here we consider another regularization method which has been subsequently developed in \cite{SHHC15,CHKS15,HRS15,HRS16} and used also in \cite{RSH18,HRS19}. In these recent papers, this method has been
	called either the ``indirect incremental method" or the ``penalization method". Below, we generalize, results of \cite{HRS15,HRS16} and show that some of these can be established in a simpler way. Within this section it is assumed that the conditions (A1)--(A4) from Section \ref{sec_analysis} hold.

To regularize the functional $\mathcal J$ defined by \eqref{J0} we introduce the functional
\begin{equation}
\mathcal J_\alpha\colon Y\rightarrow \mathbb R,\qquad \mathcal J_\alpha(y):=\max_{x\in P}\ \left\{a(x,y)-\frac{1}{2\alpha}\|x\|_X^2\right\},
\label{J_alpha}
\end{equation}
where $\alpha>0$ is a given parameter.
It is easy to see that $\mathcal J_\alpha$ is convex and {\it finite-valued} in $Y$ (unlike the functional $\mathcal J$) and $\mathcal J_{\alpha_1}\leq\mathcal J_{\alpha_2}\leq \mathcal J$ for any $\alpha_1,\alpha_2>0$, $\alpha_1\leq\alpha_2$. 

\begin{lemma}
	Let $\mathcal J$ and $\mathcal J_\alpha$ be defined by \eqref{J0} and \eqref{J_alpha}. Then
	\begin{equation}
	\lim_{\alpha\rightarrow+\infty}\mathcal J_\alpha(y)=\mathcal J(y)\quad\forall y\in Y.
	\label{J_alpha_lim}
	\end{equation}
\end{lemma}

\begin{proof}
	Let $y\in Y$ be fixed. As mentioned above, the sequence $\{\mathcal J_\alpha(y)\}_\alpha$ is nondecreasing. Therefore, it has a limit which is less than or equal to $\mathcal J(y)$. On the other hand, 
	$$\lim_{\alpha\rightarrow+\infty}\mathcal J_\alpha(y)\geq \lim_{\alpha\rightarrow+\infty}\left\{a(x,y)-\frac{1}{2\alpha}\|x\|_X^2\right\}=a(x,y)\quad\forall x\in P.$$
	Thus \eqref{J_alpha_lim} holds.
\end{proof}

The regularization of the primal problem \eqref{dual_problem} with respect to the parameter $\alpha$ defines the function $\psi:\mathbb R_+\rightarrow\mathbb R_+$ : 
\begin{equation}
\psi(\alpha):=\inf_{\substack{y\in Y\\ L(y)=1}}\mathcal J_\alpha(y),\qquad \alpha>0.
\label{psi}
\end{equation}
In view of \eqref{J_alpha} and \cite[Proposition VI 2.3]{ET74}, it holds
\begin{equation}
\psi(\alpha)=\inf_{\substack{y\in Y\\ L(y)=1}}\max_{x\in P}\ \left\{a(x,y)-\frac{1}{2\alpha}\|x\|_X^2\right\}=\max_{x\in P}\inf_{\substack{y\in Y\\ L(y)=1}}\ \left\{a(x,y)-\frac{1}{2\alpha}\|x\|_X^2\right\}.
\label{psi2}
\end{equation}
Thus, the main duality relation holds without any gap, unlike the original primal-dual problem \eqref{duality_problem}. The properties of the function $\psi$ are set out in the following theorem.

\begin{theorem}
	The function $\psi$ is continuous, nondecreasing and 
	\begin{equation}
	\lim_{\alpha\rightarrow+\infty}\psi(\alpha)=\lambda^*\leq\zeta^*,
	\label{psi_lim}
	\end{equation}
	where $\lambda^*$ and $\zeta^*$ are defined by \eqref{primal_problem} and \eqref{dual_problem}, respectively.
	\label{theorem_psi}	
\end{theorem}

\begin{proof}
	From the properties of $\{\mathcal J_\alpha\}_{\alpha>0}$, it is easy to see that $\psi$ is nondecreasing and thus it has a limit as $\alpha\rightarrow+\infty$. Comparing \eqref{primal_problem} and \eqref{psi2}$_3$ we see that $\psi(\alpha)\leq\lambda^*$. In addition, for any $x\in P$ we have
	$$\lim_{\alpha\rightarrow+\infty}\psi(\alpha)\stackrel{\eqref{psi2}}{\geq}\lim_{\alpha\rightarrow+\infty}\inf_{\substack{y\in Y\\ L(y)=1}}\ \left\{a(x,y)-\frac{1}{2\alpha}\|x\|_X^2\right\}=\inf_{\substack{y\in Y\\ L(y)=1}}\ a(x,y).$$
	Making use of the definition of $\lambda^*$, we arrive at \eqref{psi_lim}.
	
	Let $\beta>\alpha$. Since  $0_X\in P$, we have $(\alpha/\beta) x\in P$ if $x\in P$.
	Hence,
	$$\psi(\alpha)\stackrel{\eqref{psi2}}{\geq}\inf_{\substack{y\in Y\\ L(y)=1}}\max_{x\in P}\ \left\{a((\alpha/\beta) x,y)-\frac{1}{2\alpha}\|(\alpha/\beta) x\|_X^2\right\}=\frac{\alpha}{\beta} \psi(\beta).$$
	This relation and the monotonicity of $\psi$ imply
	\begin{equation*}
	\frac{\alpha}{\beta} \psi(\beta)\leq \psi(\alpha)\leq \psi(\beta).
	\label{alpha_beta}
	\end{equation*}
	Hence,
	\begin{equation}
	\limsup_{\beta\searrow\alpha}\psi(\beta)=\limsup_{\beta\searrow\alpha}\frac{\alpha}{\beta} \psi(\beta)\leq \psi(\alpha)\leq \liminf_{\beta\searrow\alpha}\psi(\beta).
	\label{alpha_beta2}
	\end{equation}
	
	Let $\beta<\alpha$. By interchanging $\alpha$ and $\beta$ in \eqref{alpha_beta}, we obtain
	\begin{equation*}
	\frac{\beta}{\alpha} \psi(\alpha)\leq \psi(\beta)\leq \psi(\alpha)\quad\mbox{or}\quad \psi(\beta)\leq \psi(\alpha)\leq \frac{\alpha}{\beta} \psi(\beta).
	\label{beta_alpha}
	\end{equation*}
	Hence,
	\begin{equation}
	\limsup_{\beta\nearrow\alpha}\psi(\beta)\leq \psi(\alpha)\leq \limsup_{\beta\nearrow\alpha}\frac{\alpha}{\beta} \psi(\beta)=\liminf_{\beta\nearrow\alpha}\psi(\beta).
	\label{beta_alpha2}
	\end{equation}
	From \eqref{alpha_beta2} and \eqref{beta_alpha2}, we have
	$$\limsup_{\beta\rightarrow\alpha}\psi(\beta)\leq \psi(\alpha)\leq \liminf_{\beta\rightarrow\alpha}\psi(\beta)\quad\mbox{or}\quad \lim_{\beta\rightarrow\alpha}\psi(\beta)= \psi(\alpha),$$
	implying the continuity of $\psi$.
\end{proof}

It is worth noting that for any value of $\alpha>0$, the quantity $\psi(\alpha)$ is a lower bound of $\lambda^*$ and $\zeta^*$. Upper bounds of $\lambda^*$ and $\zeta^*$ will be derived in the next section.

From the numerical point of view, it is useful if the functional $\mathcal J_\alpha$ is differentiable in the G\^ateaux sense. Below we establish this property of the regularized functional.

\begin{lemma}
	Let $X$ be a Hilbert space with the scalar product $(.,.)_X$ and define 
	\begin{equation}
	\Pi_\alpha\colon Y\rightarrow P,\qquad \Pi_\alpha y:=\mathrm{arg}\max_{x\in P}\ \left\{a(x,y)-\frac{1}{2\alpha}\|x\|_X^2\right\}.
	\label{Pi_alpha}
	\end{equation}	
	Then $\Pi_\alpha$ is Lipschitz continuous in $Y$ and
	\begin{equation}
	\mathcal J'_\alpha(y;z):=\lim_{t\rightarrow0}\frac{1}{t}[\mathcal J_\alpha(y+tz)-\mathcal J_\alpha(y)]=a(\Pi_\alpha y,z)\quad\forall\alpha>0,\;\;\forall y,z\in Y.
	\label{J_deriv}
	\end{equation}
	\label{lem_G-deriv}
\end{lemma}

\begin{proof}
	Since $X$ is a Hilbert space, it is easy to see that there exists a unique $\Pi_\alpha y$ solving \eqref{Pi_alpha} and satisfying the variational inequality
	\begin{equation}
	\frac{1}{\alpha}(\Pi_\alpha y, x-\Pi_\alpha y)_X\geq a(x-\Pi_\alpha y, y)\quad \forall x\in P,
	\;\;\forall y\in Y.
	\label{Pi_alpha2}
	\end{equation}
	Hence, we derive the inequalities
	\begin{align*}
	\frac{1}{\alpha}(\Pi_\alpha y, \Pi_\alpha (y+tz)-\Pi_\alpha y)_X &\geq a(\Pi_\alpha (y+tz)-\Pi_\alpha y, y),\\
	\frac{1}{\alpha}(\Pi_\alpha (y+tz), \Pi_\alpha y-\Pi_\alpha (y+tz))_X &\geq a(\Pi_\alpha y-\Pi_\alpha (y+tz), y+tz),
	\end{align*}
	which hold for any $y,z\in Y$ and any $t\in\mathbb R$. By adding these inequalities, we obtain
	$$\frac{1}{\alpha}\|\Pi_\alpha (y+tz)-\Pi_\alpha y\|_X^2\leq ta(\Pi_\alpha (y+tz)-\Pi_\alpha y,z)\leq t\|a\|\|\Pi_\alpha (y+tz)-\Pi_\alpha y\|_X\|z\|_Y.$$
	Thus $\Pi_\alpha$ is Lipschitz continuous in $Y$.
	
	From \eqref{J_alpha} and \eqref{Pi_alpha} we have, for any $t\in \mathbb R$ and any $y,z\in Y$,
	$$\mathcal J_\alpha(y)=a(\Pi_\alpha y,y)-\frac{1}{2\alpha}\|\Pi_\alpha y\|_X^2\geq a(\Pi_\alpha (y+tz),y)-\frac{1}{2\alpha}\|\Pi_\alpha (y+tz)\|_X^2,$$
	$$\mathcal J_\alpha(y+tz)=a(\Pi_\alpha (y+tz),y+tz)-\frac{1}{2\alpha}\|\Pi_\alpha (y+tz)\|_X^2\geq a(\Pi_\alpha y,y+tz)-\frac{1}{2\alpha}\|\Pi_\alpha y\|_X^2.$$
	Hence,
	$$a(\Pi_\alpha y,z)\leq\frac{1}{t}[\mathcal J_\alpha(y+tz)-\mathcal J_\alpha(y)]\leq a(\Pi_\alpha (y+tz),z),$$
	proving \eqref{J_deriv}.
\end{proof}

%\begin{remark}
%\emph{From the proof of Lemma \ref{lem_G-deriv}, it seems that the additional assumption on $X$ could be generalized. It suffices to be the function $\Pi_\alpha$ uniquely defined and continuous. }
%\end{remark}

Using the differentiability of $\mathcal J_\alpha$, one can rewrite the problem \eqref{psi} as a system of nonlinear variational equations.
\begin{theorem}
	Let $X$ be a Hilbert space with the scalar product $(.,.)_X$ and let $y_\alpha$ be a minimizer in \eqref{psi}. Then there exists $\lambda_\alpha\in\mathbb R_+$ such that the pair $(y_\alpha,\lambda_\alpha)$ is a solution of the system:
	\begin{equation}
	\left.\begin{array}{c}
	a(\Pi_\alpha y_\alpha,z)=\lambda_\alpha L(z)\quad\forall z\in Y,\\
	L(y_\alpha)=1.
	\end{array}\right\}
	\label{system}
	\end{equation}
	\label{theorem_G-deriv}
	Conversely, if $(y_\alpha,\lambda_\alpha)$ is a solution to \eqref{system} then $y_\alpha$ solves \eqref{psi}.
\end{theorem}

\begin{remark}
	\emph{In \cite{HRS15}, the function $\tilde\psi\colon\alpha\mapsto \lambda_\alpha$ was introduced and analysed for the case of Hencky plasticity. It is worth noticing that this function is well defined even if \eqref{psi} does not have a minimizer in $Y$. In addition, $\tilde\psi$ is continuous and nondecreasing, with $\psi(\alpha)\leq\tilde\psi(\alpha)\leq\lambda^*$ for any $\alpha>0$, and $\tilde\psi(\alpha)\rightarrow\lambda^*$ as $\alpha\rightarrow+\infty$. One can expect that these considerations from \cite{HRS15} can be extended to our abstract problem.}
\end{remark}

%%%%%%%%%%%%%%%%%%%%%%

\section{A computable majorant of $\zeta^*$}
\label{sec_majorant}

For classical limit analysis problems, computable majorants of $\zeta^*$ have been derived in \cite{RSH18, HRS19}. The aim of this section is to derive a more general majorant valid for the abstract problem \eqref{dual_problem}. In our analysis, we shall use the assumptions (A1)--(A4) and (D1)--(D4) of Theorem \ref{theorem_duality3}. The following alternative to the assumption (D3) will also be considered:
\begin{itemize}
	\item[(D3$^\prime$)] $P=P_A+P_C=\{x\in X\ |\; x=x_A+x_C,\; x_A\in P_A,\; x_C\in P_C\}$, where $P$, $P_A$ and $P_C$ have the same properties as in (D3).
\end{itemize}
We note that (D3$^\prime$) is more restrictive than (D3); it has been used in \cite{HRS19}.

From the definition of $\zeta^*$ (see \eqref{primal_problem}), we have the following simple upper bound of $\zeta^*$:
\begin{equation}
\zeta^*\leq\frac{\mathcal J(y)}{L(y)}\qquad \forall y\in Y,\;\;y\in\mathrm{dom}\,\mathcal J, \;\;L(y)>0.
\label{upper_bound1}
\end{equation}
Unfortunately, if the set $P$ is unbounded then it is difficult or even impossible to find $y\in\mathrm{dom}\,\mathcal J$ in such a way that the bound \eqref{upper_bound1} would be sufficiently sharp. The aim of this section is to derive an upper bound of $\zeta^*$ for a larger class of functions $y\in Y$, not necessarily belonging to $\mathrm{dom}\,\mathcal J$. 

First, we need to characterize the set $\mathrm{dom}\,\mathcal J$. For this purpose, we define the closed convex cone
\begin{equation}
\mathcal K:=\{y\in Y\ |\;\; a(x,y)\leq0\;\;\forall x\in P_C\},
\label{K}
\end{equation}
and the convex, finite-valued functional
\begin{equation}
\mathcal J_A\colon Y\rightarrow \mathbb R,\quad \mathcal J_A(y):=\max_{x\in P_A}\,a(x,y),\quad y\in Y.
\label{J_A}
\end{equation}

\begin{lemma}
Let the assumptions (A1)--(A4) and (D1)--(D4) be satisfied. Then
\begin{equation}
\mathrm{dom}\,\mathcal J=\mathcal K\quad\mbox{and}\quad \mathcal J(y)\leq \mathcal J_A(y)\quad\forall y\in\mathcal K.
\label{J_bound}
\end{equation}
Moreover, if (D3$\,^\prime$) holds then $\mathcal J(y)= \mathcal J_A(y)$ for any $y\in\mathcal K$.
\label{lem_J}
\end{lemma}

\begin{proof}
Assume that $y\not\in\mathcal K$. Then there exists $x_C\in P_C$ such that $a(x_C,y)>0$. From (D3), it follows that $\alpha x_C\in P$ for any $\alpha\geq0$. Hence,
$$\mathcal  J(y)\geq\lim_{\alpha\rightarrow+\infty}a(\alpha x_C,y)=\lim_{\alpha\rightarrow+\infty}\alpha a(x_C,y)=+\infty.$$
Let $y\in\mathcal K$. Then
\begin{equation}
\mathcal  J(y)\leq \sup_{x_A\in P_A}a(x_A,y)+\sup_{x_C\in P_C}a(x_C,y)=J_A(y)+0=J_A(y)<+\infty.
\label{J<=J_A}
\end{equation}
If (D3$\,^\prime$) holds then $P_A=P_A+\{0_X\}\subset P$. Hence,
\begin{equation}
\mathcal  J(y)\geq \sup_{x\in P_A}a(x,y)=J_A(y).
\label{J>=J_A}
\end{equation}
From \eqref{J<=J_A} and \eqref{J>=J_A}, it follows that $\mathcal  J(y)= J_A(y)$ for any $y\in\mathcal K$.
\end{proof}

From the definition of $\mathcal J_A$ and the boundedness of $a$ and $P_A$, we easily derive the useful estimates
\begin{equation}
|\mathcal J_A(y_1)-\mathcal J_A(y_2)|\leq \mathcal J_A(y_1-y_2),\quad \forall y_1,y_2\in Y,
\label{J_A_est1}
\end{equation}
and
\begin{equation}
\mathcal J_A(y)\leq \varrho_A\|a\|\|y\|_Y,\quad \forall y\in Y,\quad \varrho_A:=\max_{x\in P_A}\|x\|_X.
\label{J_A_est2}
\end{equation}

In order to estimate $\zeta^*$ using $y\not\in \mathcal K$, it is important to measure the distance between $y$ and $\mathcal K$. Define the quantity
\begin{equation}
\|\Pi_C\,y\|_X:=\left(\max_{x\in P_C}\{-\|x\|^2_X+2a(x,y)\}\right)^{1/2},\quad y\in Y.
\label{Pi_C_norm}
\end{equation}
\begin{remark}
	\emph{The notation $\|\Pi_C\,y\|_X$ including the norm in $X$ is justified if $X$ is a Hilbert space. Indeed, define the operator
		\begin{equation}
		\Pi_C\colon Y\rightarrow P_C,\quad \Pi_C\,y:=\mathrm{arg}\max_{x\in P_C}\{-\|x\|^2_X+2a(x,y)\},\quad y\in Y.
		\label{Pi_C}
		\end{equation}
		From the cone property of $P_C$, \eqref{Pi_C} is equivalent to 
		$$\|\Pi_C\,y\|_X^2=a(\Pi_C\,y,y)\quad\mbox{and}\quad (\Pi_C\,y,x)\geq a(x,y)\quad\forall x\in P_C.$$
		Hence, we obtain \eqref{Pi_C_norm}. 
	}
\end{remark}
It is also useful to note that if $y\in\mathcal K$ then $\|\Pi_C\,y\|_X=0$. We have the following result.
\begin{lemma}
	Let the assumptions (A1)--(A4) and (D1)--(D4) be satisfied and $c^*>0$, $\mathcal K$, $\|\Pi_C\,y\|_X$ be defined by \eqref{inf-sup_abstract}, \eqref{K}, and \eqref{Pi_C_norm}, respectively.
	Then
	\begin{equation}
	\min_{z\in\mathcal K}\|y-z\|\leq C_*\|\Pi_C\,y\|_X,\quad \forall y\in Y, \quad C_*:=c_*^{-1}>0.
	\label{distance}
	\end{equation}
\label{lem_distance}
\end{lemma}

\begin{proof}
Using \eqref{K}, \cite[Proposition VI 2.3]{ET74} and the substitution $z\mapsto z+y$, we consequently derive
\begin{align}
\min_{z\in\mathcal K}\|y-z\|^2&=\min_{z\in Y}\,\sup_{x\in P_C}\left\{\|y-z\|^2+2a(x,z)\right\}\nonumber\\
& =\sup_{x\in P_C}\,\min_{z\in Y}\left\{\|y-z\|^2+2a(x,z)\right\}\nonumber\\
& =\sup_{x\in P_C}\,\min_{z\in Y}\left\{\|z\|^2+2a(x,z)+2a(x,y)\right\} \quad\forall y\in Y.
\label{dist1}
\end{align}
For any $x\in X$, there exists a unique $z_x\in Y$ such that
$$(z_x,z)_Y=-a(x,z)\quad \forall z\in Y.$$
Hence,
\begin{equation}
\|z_x\|_X=\sup_{\substack{z\in Y\\ z\neq 0_Y}}\frac{a(x,z)}{\|z\|_Y}\quad\mbox{and}\quad \min_{z\in Y}\left\{\|z\|^2+2a(x,z)\right\}=-\|z_x\|^2.
\label{z_x}
\end{equation}
Inserting \eqref{z_x} into \eqref{dist1}, we find that
\begin{align}
\min_{z\in\mathcal K}\|y-z\|^2&=\sup_{x\in P_C}\,\min_{z\in Y}\left\{\|z\|^2+2a(x,z)+2a(x,y)\right\}\nonumber\\
&=\sup_{x\in P_C}\,\left\{-\left(\sup_{\substack{z\in Y\\ z\neq 0_Y}}\frac{a(x,z)}{\|z\|_Y}\right)^2+2a(x,y)\right\}\nonumber\\
&\stackrel{\eqref{inf-sup_abstract}}{\leq}\sup_{x\in P_C}\,\left\{-c_*^2\|x\|_X^2+2a(x,y)\right\}\nonumber\\
&=\max_{x\in P_C}\,\left\{-c_*^2\|x/c_*^2\|_X^2+2a(x/c_*^2,y)\right\}\nonumber\\
&=\frac{1}{c_*^2}\max_{x\in P_C}\,\left\{-\|x\|_X^2+2a(x,y)\right\}\stackrel{\eqref{Pi_C}}{=}C_*^2\|\Pi_C\,y\|_X^2 \qquad\forall y\in Y,
\label{dist2}
\end{align}
which gives the desired result.
\end{proof}

Using Lemma \ref{lem_J} and \ref{lem_distance}, we derive the following upper bound of $\zeta^*$.
\begin{theorem}
Let the assumptions (A1)--(A4) and (D1)--(D4) be satisfied and $y\in Y$ be such that 
\begin{equation}
L(y)>C_*\|\Pi_C\,y\|_X\|L\|_{Y^*}.
\label{L_assump}
\end{equation}
Then
\begin{equation}
\zeta^*\leq\frac{\mathcal J_A(y)+\varrho_AC_*\|a\|\|\Pi_C\,y\|_X}{L(y)-C_*\|\Pi_C\,y\|_X\|L\|_{Y^*}}.
\label{upper_bound2}
\end{equation}
\end{theorem}	

\begin{proof}
Let $y\in Y$ satisfy \eqref{L_assump}. By Lemma \ref{lem_distance} there exists $z_y\in \mathcal K$ such that
\begin{equation}
\|y-z_y\|_Y\leq C_*\|\Pi_C\,y\|_X.
\label{distance2}
\end{equation}
For any $\lambda>\frac{\mathcal J_A(y)+\varrho_AC_*\|a\|\|\Pi_C\,y\|_X}{L(y)-C_*\|\Pi_C\,y\|_X\|L\|_{Y^*}}$, we have
\begin{align*}
\mathcal J(z_y)-\lambda L(z_y) & \stackrel{\eqref{J_bound}}{\leq}\mathcal J_A(z_y)-\lambda L(z_y)\nonumber\\
& =\mathcal J_A(y)-\lambda L(y)+[\mathcal J_A(z_y)-\mathcal J_A(y)]+\lambda L(y-z_y)\nonumber\\
& \stackrel{\eqref{J_A_est1},\eqref{J_A_est2}}{\leq}\mathcal J_A(y)-\lambda L(y)+(\varrho_A\|a\|+\lambda\|L\|_{Y^*})\|y-z_y\|_Y\nonumber\\
& \stackrel{\eqref{distance2}}{\leq}\mathcal J_A(y)-\lambda L(y)+C_*(\varrho_A\|a\|+\lambda\|L\|_{Y^*})\|\Pi_C\,y\|_X\nonumber\\
&=\mathcal J_A(y)+\varrho_AC_*\|a\|\|\Pi_C\,y\|_X-\lambda\left[L(y)-C_*\|L\|_{Y^*}\|\Pi_C\,y\|_X\right]<0.
\end{align*}
Hence, $L(z_y)>\mathcal J(z_y)/\lambda\geq 0$ and
$$\zeta^*\stackrel{\eqref{upper_bound1}}{\leq}\frac{\mathcal J(z_y)}{L(z_y)}<\lambda\qquad \forall\lambda>\frac{\mathcal J_A(y)+\varrho_AC_*\|a\|\|\Pi_C\,y\|_X}{L(y)-C_*\|\Pi_C\,y\|_X\|L\|_{Y^*}}.$$
This implies \eqref{upper_bound2}.
\end{proof}

\begin{remark}
\emph{If the assumption (D3$^\prime$) holds and $y\in\mathcal K$ then $\mathcal J_A(y)=\mathcal J(y)$, $\|\Pi_C\,y\|_X=0$, and thus the bounds \eqref{upper_bound1} and \eqref{upper_bound2} coincide.}
\end{remark}

\begin{remark}
	\emph{If $y\in Y$ is sufficiently close to the cone $\mathcal K$ then the assumption \eqref{L_assump} is satisfied. This can be achieved by a convenient numerical method, e.g., by the regularization method presented in the previous section.}
\end{remark}

\begin{remark}
	\emph{The bound \eqref{upper_bound2} is computable if estimates of $\|L\|_{Y^*}$, $\|\Pi_C\,y\|_X$ and $C_*$ are at our disposal. The computable bounds of $\|L\|_{Y^*}$, $\|\Pi_C\,y\|_X$ are available in the literature on a posteriori error analysis. Computable bounds of the inf-sup constant $C_*$ have appeared in the literature quite recently, see \cite{HRS19} and references therein.}
\end{remark}

\begin{remark}
	\emph{In \cite{RSH18}, a computable majorant of the limit load was used in the Hencky plasticity problem to prove convergence of the standard finite element method and to detect locking effects that may arise when the simplest P1 elements are used.}
\end{remark}

%%%%%%%%%%%%%%%%%%%%%%%%%%%%%%%%%%

\section{Examples}
\label{sec_examples}

In this section, we illustrate the abstract problem \eqref{duality_problem} on particular examples from nonlinear mechanics and discuss the validity of the assumptions (A1)--(A4), (B) and (D1)--(D4) presented in Section \ref{sec_analysis}. In all examples we consider a bounded domain $\Omega\subset\mathbb R^d$, $d=2,3$, with Lipschitz continuous boundary $\partial\Omega$. The outward unit normal to $\partial\Omega$ is denoted by $\nu$. The abstract spaces $X$ and $Y$ will be represented by $L^2$ and $H^1$ spaces, respectively, for the sake of simplicity.

\subsection{Limit analysis in classical perfect plasticity}
\label{subsec_ex1}

Details of the mathematical theory of limit analysis in classical perfect plasticity may be found in \cite{T85} or \cite{Ch96}. For its engineering applications we refer, for example, to \cite{CL90, Sl13}. The aim is to find the largest load factor at which plastic behaviour may be sustained, in the context of proportional loading. We briefly recapitulate results presented in \cite{HRS19, RSH18, HRS16, HRS15}. 

A body occupying the domain $\Omega$ is fixed on a part $\Gamma_0\subset \partial\Omega$ and surface forces $f\colon\Gamma_f\rightarrow \mathbb R^d$ act on the remaining part $\Gamma_f$ of $\partial\Omega$. We assume that $\Gamma_0$ and $\Gamma_f$ have a positive surface measure. Let $F\colon\Omega\rightarrow \mathbb R^d$ denote the volume force. The external loads are parametrized by a scalar factor $\lambda\geq0$. 
	
Next, we denote the space of symmetric matrices (second order tensors) by $\mathbb R^{d\times d}_{sym}$. The Cauchy stress field $\sigma\colon\Omega\rightarrow\mathbb R^{d\times d}_{sym}$ satisfies the equilibrium equation and traction boundary condition
\begin{subequations}
	\begin{align}
	\mathrm{div}\,\sigma+\lambda F=0 &\quad \mbox{in}\;\Omega, \label{balance1a}\\
	\sigma\nu=\lambda f & \quad \mbox{on}\;\Gamma_f, \label{balance1b}
	\end{align}
\end{subequations}
and is plastically admissible in the sense that 
\begin{equation}
\sigma\in B\;\;\mbox{in}\;\Omega,\quad B:=\{\tau\in\mathbb R^{d\times d}_{sym}\ |\;\; \varphi(\tau)\leq 0\}.
\label{B}
\end{equation}
Here, $\varphi\colon \mathbb R^{d\times d}_{sym}\rightarrow \mathbb R$, $\varphi(0)<0$, is a convex function representing a yield criterion. For the sake of simplicity, we assume that $\varphi$ and thus $B$ are independent of the spatial variable.

The infinitesimal strain rate $\varepsilon\colon\Omega\rightarrow\mathbb R^{d\times d}_{sym}$ and the displacement rate $v\colon\Omega\rightarrow\mathbb R^d$ satisfy the relations
\begin{equation}
\varepsilon:=\varepsilon(v)=\frac{1}{2}[\nabla v+(\nabla v)^\top] \;\;\mbox{in}\;\Omega,\quad v=0\;\;\mbox{on}\;\Gamma_0.
\label{u}
\end{equation}
The last ingredient of the perfectly plastic model is a plastic flow rule that relates $\sigma$ and $\varepsilon$, and which is based on the set $B$. This relation is represented by the principle of maximum plastic dissipation in quasistatic models or by a generalized projection of $\mathbb R^{d\times d}_{sym}$ onto $B$ in total strain models. We skip its definition, for the sake of brevity.

Formally, the limit load factor $\lambda^*$ is defined as the supremum over $\lambda\geq 0$ subject to \eqref{balance1a}, \eqref{balance1b} and \eqref{B}. To define $\lambda^*$ more precisely and in the form \eqref{primal_problem}, it is necessary to introduce a convenient function space $X$ for stress fields. For this purpose define the Hilbert space
$$X:=L^2(\Omega;\mathbb R^{d\times d}_{sym})=\{\sigma\colon \Omega\rightarrow\mathbb R^{d\times d}_{sym}\ |\; \sigma_{ij}\in L^2(\Omega),\;\; i,j=1,2,\ldots d\}$$
equipped with the scalar product and norm
$$(\sigma,\varepsilon)_X:=(\sigma,\varepsilon)_2=\int_\Omega\sigma:\varepsilon\,dx,\quad \|\sigma\|_X:=\|\sigma\|_2=\sqrt{(\sigma,\sigma)_2},$$
where $\sigma:\varepsilon=\sigma_{ij}\varepsilon_{ij}$ with the summation convention on repeated indices.
The corresponding primal space $Y$ is chosen as follows:
$$Y:=\{v\in W^{1,2}(\Omega;\mathbb R^d)\ |\;\; v=0\;\mbox{a.e. in } \Gamma_0\}.$$
It is also a Hilbert space representing rates of displacements with the following scalar product and norm:
$$(u,v)_Y:=(\nabla u, \nabla v)_2,\quad \|v\|_Y:=\|\nabla v\|_2.$$ 

Using the spaces $X,Y$ and Green's theorem, a weak formulation of \eqref{balance1a} and \eqref{balance1b} for fixed $\sigma$ reads as follows:
\begin{equation}
a(\sigma,v)=\lambda L(v)\quad\forall v\in Y,
\label{balance2}
\end{equation}
where 
\begin{equation}
a(\sigma,v):=\int_\Omega\,\sigma:\varepsilon(v)\,dx,\quad L(v):=\int_\Omega F\cdot v\,dx+\int_{\Gamma_f}f\cdot v\,ds,\quad v\in Y,
\label{balance3}
\end{equation}
with $\sigma\in X$, $F\in L^2(\Omega;\mathbb R^d)$ and $f\in L^2(\Gamma_f;\mathbb R^d)$.
It is easy to see that $a$ is a continuous bilinear form in $X\times Y$ and $L\in Y^*$. Using the notation from Section \ref{sec_intro}, one can write
$$\lambda^*=\sup\{\lambda\in\mathbb R_+\ |\;\; P\cap \Lambda_\lambda\neq\emptyset\}=\sup_{\sigma\in P}\inf_{\substack{v\in Y\\ L(v)=1}}\ a(\sigma,v),$$
where 
\begin{equation}
P:=\{\sigma\in X\ |\;\; \sigma\in B\;\;\mbox{a.e. in } \Omega\},\quad \Lambda_\lambda:=\{\sigma\in X\ |\;\; a(\sigma,v)=\lambda L(v)\;\;\forall v\in Y\}.
\label{P}
\end{equation}
The sets $P$ and $\Lambda_\lambda$ are closed, convex and non-empty in $X$ and represent plastically and statically admissible stresses, respectively.

We note that the set $P$ is defined in a pointwise sense. Consequently, the sets $P_A$, $P_C$ and the functions $\mathcal J$, $\mathcal J_\alpha$, $\Pi_\alpha$ and $\Pi_C$ introduced in the previous sections may be also defined in a pointwise sense.
%Appropriate choices of $P_A$ and $P_C$ depend on a particular yield function $\varphi$. They have been introduced in \cite{HRS19} for basic yield criteria. Unlike \cite{HRS19}, we do not strictly require that $P=P_A+P_C$. It enables to extend the results for more advanced yield criteria.
To illustrate, we choose the von Mises yield criterion defined by 
\begin{equation}
\varphi(\sigma):=|\sigma^D|-\gamma,\quad \gamma>0,\; \sigma^D=\sigma-\frac{1}{d}(\mathrm{tr}\,\sigma)I, \; |\sigma|:=\sqrt{\sigma_{ij}\sigma_{ij}},
\label{vM}
\end{equation}
where $I$ is the unit $d\times d$ matrix, $\mathrm{tr}\,\sigma$ denotes the trace of $\sigma$, $\sigma^D$ is the deviatoric part of $\sigma$ and $\gamma>0$ is a given parameter representing the initial yield stress.
From \cite{T85, Ch96, HRS19}, it is known that $P$ can be decomposed according to $P=P_A+P_C$, where
$$P_A=\{\tau\in X\ |\;\; |\tau|\leq \gamma\;\;\mbox{a.e. in } \Omega\},\quad P_C=\{\tau\in X\ |\;\; \exists q\in L^2(\Omega):\;\; \tau=qI\}.$$
Clearly, $P_A$ is bounded in $X$ and $P_C$ is a closed subspace of $X$, that is, a convex cone. To prove \eqref{inf-sup_abstract}, we use the well-known inf-sup condition for incompressible flow media with $c_\Omega>0$:
\begin{equation}
\inf_{\substack{x_C\in P_{C}\\ x_C\neq0_X}}\ \sup_{\substack{y\in Y\\ y\neq0_Y}}\ \frac{a(x_C,y)}{\|x_C\|_X\|y\|_Y} =\inf_{\substack{\tau\in P_{C}\\ \tau\neq0_X}}\ \sup_{\substack{v\in Y\\ v\neq0_Y}}\ \frac{\int_\Omega\,\tau:\varepsilon(v)\,dx}{\|\tau\|_2\|\nabla v\|_2}=\frac{1}{\sqrt{d}}\inf_{\substack{q\in L^2(\Omega)\\ q\neq0}}\ \sup_{\substack{v\in Y\\ v\neq0_Y}}\ \frac{\int_\Omega q\,\mathrm{div}\,v\,dx}{\|q\|_2\|\nabla v\|_2}\geq \frac{c_\Omega}{\sqrt{d}}.
\label{inf-sup_vM}
\end{equation}
Thus, the condition \eqref{inf-sup_abstract} holds with $c_*=c_\Omega/\sqrt{d}$. Consequently, the assumptions (A1)--(A4), (D1)--(D4) from Section \ref{sec_analysis} are satisfied and from Theorem \ref{theorem_duality3} it follows that
$$\lambda^*=\zeta^*=\inf_{\substack{v\in Y\\ L(v)=1}}\sup_{\sigma\in P}\ a(\sigma,v)=\inf_{\substack{v\in Y\\ L(v)=1}}\mathcal J(v).$$
Notice that if $\Gamma_0=\partial\Omega$ then it is necessary to use Theorem \ref{theorem_duality6} with the weaker assumption \eqref{inf-sup_abstract4} instead of (D4). In this case, we replace the space $L^2(\Omega)$ in \eqref{inf-sup_vM} by $L^2_0(\Omega)=\{q\in L^2(\Omega)\ |\; \int_\Omega q\,dx=0\}$, see \cite{RSH18,HRS19}.

The primal functional $\mathcal J$ for the von Mises yield criterion is given by
$$\mathcal J(v)=\sup_{\sigma\in P}\ a(\sigma,v)=\left\{
\begin{array}{ll}
\displaystyle \int_\Omega \gamma|\varepsilon(v)|\,dx,& \mathrm{div}\,v=0\;\mbox{in } \Omega,\\
+\infty, & \mbox{otherwise},% \mathrm{div}\,v\neq0\;\mbox{in } \Omega,
\end{array}
\right.\qquad\forall v\in Y.$$
This functional may have no minimizers in $Y$. To guarantee that the primal problem is solvable, it is necessary to use another choice of $X$ and $Y$, as was done, for example, in \cite{Ch80,T85,Ch96}. In particular, the assumptions (C1)--(C3) of Theorem \ref{theorem_duality2} were verified in \cite{Ch80,Ch96}.

The functions $\mathcal J_\alpha$, $\mathcal J_A$ and $\Pi_C$ for the von Mises yield criterion can be found in the following forms:
\begin{equation*}
\mathcal J_\alpha(v):=\int_\Omega j_\alpha(\varepsilon(v))\,dx,\quad j_\alpha(\varepsilon)=\left\{
\begin{array}{cl}
\frac{1}{2}\alpha|\varepsilon|^2,& \alpha|\varepsilon^D|\leq\gamma\\[2pt]
\frac{1}{2d}\alpha(\mathrm{tr}\,\varepsilon)^2+\gamma|e^D|-\frac{\gamma^2}{2\alpha},& \alpha|e^D|\geq\gamma,
\end{array}
\right.,
\label{J_alpha_VM}
\end{equation*}
$$\mathcal J_A(v)=\int_\Omega \gamma|\varepsilon(v)|\,dx,\quad \|\Pi_C\,v\|_2=d^{-1/2}\|\mathrm{div}\,v\|_2\quad\forall v\in Y.$$
Let us recall that they are important for the regularization method and the computable majorant presented in the previous sections. We refer to \cite{HRS15, HRS16, RSH18, HRS19} for more details.

\begin{remark}
\emph{If we choose the Drucker-Prager or Mohr-Coulomb yield criteria in \eqref{B} instead of von Mises then it is also possible to find an appropriate split $P=P_A+P_C$ such that the assumptions (D3) and even (D3$^\prime$) are satisfied. But for these criteria the cone $P_C$ is not a subspace of $X$. Therefore, it is necessary to work with the inf-sup condition on convex cones, see \cite{HRS19}.}
\end{remark}

\subsection{Plastically admissible stresses in strain-gradient plasticity}
\label{subsec_ex2}

In the next two subsections, we consider as further examples the models of strain-gradient plasticity presented in \cite{Reddy_etal2008, Reddy2011a, CEMRS17, RS20}. First, following \cite{RS20}, we introduce a subproblem that enables us to decide whether a given stress tensor is plastically admissible or not. We note that this problem is simple in classical plasticity where the yield criterion can be verified pointwisely (see, for example the definition of $P$ in \eqref{P}). However, plastic yield criteria in strain-gradient plasticity are non-local and the verification is strongly non-trivial.
	
Beside the space $\mathbb R^{d\times d}_{sym}$ defined in Section \ref{subsec_ex1}, we also use the following spaces of the second and third order tensors, respectively:
$$\mathbb R^{d\times d}_{sym,0}:=\{\pi\in\mathbb R^{d\times d}_{sym}\ |\; \mathrm{tr}\,\pi=0\},$$
$$\mathbb R^{d\times d\times d}_{sym,0}:=\{\Pi\in\mathbb R^{d\times d\times d}\ |\; \Pi_{ijk} = \Pi_{jik}, \;i,j,k=1,2,\ldots,d, \; \Pi_{ppk} = 0,\; k=1,2,\ldots,d\}.$$
Thus, the third order tensor $\Pi$ belongs to $\mathbb R^{d\times d\times d}_{sym,0}$ if it is symmetric and deviatoric with respect to the first two indices.

We assume that $\sigma\colon\Omega\rightarrow\mathbb R^{d\times d}_{sym}$ is a given stress field and $\sigma^D\colon\Omega\rightarrow\mathbb R^{d\times d}_{sym,0}$ denotes its deviatoric part. The theory of strain gradient plasticity makes use of second- and third-order tensors $\pi\colon\Omega\rightarrow\mathbb R^{d\times d}_{sym,0}$ and $\Pi\colon\Omega\rightarrow\mathbb R^{d\times d\times d}_{sym,0}$ that represent microstresses. We say that $\sigma$ is {\it plastically admissible} if there exists a pair $(\pi,\Pi)$ such that
\begin{equation}
\sigma^D = \pi - \mbox{div}\,\Pi\quad\mbox{in }\Omega,\quad \Pi\nu= 0\;\;\mbox{on } \Gamma_F,
\label{microforce_balance}
\end{equation}
\begin{equation}
\varphi_\ell(\pi,\Pi):=\sqrt{|\pi|^2+\ell^{-2}|\Pi|^2}- \gamma\leq0\quad\mbox{in }\Omega,
\label{yield}
\end{equation}
where $\gamma>0$ is the yield stress, $\ell>0$ is the length parameter, $|\Pi|^2=\Pi\circ\Pi:=\Pi_{ijk}\Pi_{ijk}$ and $\Gamma_F\subset\partial\Omega$. The part of $\partial\Omega$ complementary to $\Gamma_F$ in $\partial\Omega$ is denoted by $\Gamma_H$. 

We note that the yield criterion \eqref{yield} can be viewed as an extension of the classical condition \eqref{vM}. Indeed, setting $\Pi=0$ we derive the sufficient condition $|\sigma^D|\leq \gamma$ for $\sigma$ to be plastically admissible. Unlike the classical case, the stress $\sigma$ can be plastically admissible even if $|\sigma^D|> \gamma$. 

If $\sigma$ is plastically admissible then $\lambda\sigma$ is also plastically admissible for any $\lambda\in[0,1]$. This parametrization motivates us to introduce the following problem: \textit{find the maximal value $\lambda^*$ of $\lambda\geq0$ for which $\lambda\sigma$ is plastically admissible in the sense of \eqref{microforce_balance} and \eqref{yield}}. Clearly, if $\lambda^*>1$ then $\sigma$ is admissible.

Let us define $\lambda^*$ more precisely, using the abstract problem \eqref{primal_problem}. We assume that all components of $\sigma$, $\pi$ and $\Pi$ belong to $L^2(\Omega)$, that is, $\sigma\in L^2(\Omega;\mathbb R^{d\times d}_{sym})$, $\pi\in L^2(\Omega;\mathbb R^{d\times d}_{sym,0})$ and $\Pi\in L^2(\Omega;\mathbb R^{d\times d\times d}_{sym,0})$. The space $X$ is defined as the space of pairs $(\pi,\Pi)$ endowed with the scalar product
$$((\pi,\Pi),(\bar\pi,\bar\Pi))_X:=\int_\Omega(\pi:\bar\pi+\Pi\circ\bar\Pi)\, dx.$$
The primal space 
$$Y:=\{q\in L^2(\Omega;\mathbb R^{d\times d}_{sym,0})\ |\;\; \nabla q\in L^2(\Omega;\mathbb R^{d\times d\times d}_{sym,0}),\;  q= 0 \;\mbox{on } \Gamma_H\}$$
is the Hilbert space of admissible plastic strain rates with the scalar product
$$(q,\bar q)_Y:=\int_\Omega (q:\bar q+\nabla q\circ\nabla \bar q)\, dx.$$
Using the spaces $X$ and $Y$, we introduce the following weak form of \eqref{microforce_balance}:
\begin{equation}
\int_\Omega [\pi:q+\Pi\circ\nabla q]\,dx=\int_\Omega \sigma^D:q\,dx\quad\forall\ q\in Y,
\label{weak_balance}
\end{equation}
and define the forms $a\colon X\times Y$ and $L\in Y^*$ by
$$a((\pi,\Pi),q):=\int_\Omega [\pi:q+\Pi\circ\nabla q]\,dx,\quad L(q):=\int_\Omega \sigma^D:q\,dx.$$
Then the dual problem \eqref{primal_problem} reads
$$\lambda^*=\sup\{\lambda\in\mathbb R_+\ |\;\; P\cap \Lambda_\lambda\neq\emptyset\}=\sup_{(\pi,\Pi)\in P}\inf_{\substack{q\in Y\\ L(q)=1}}\ a((\pi,\Pi),q),$$
where 
$$P:=\{(\pi,\Pi)\in X\ |\;\; \sqrt{|\pi|^2+\ell^{-2}|\Pi|^2}\leq \gamma\;\;\mbox{a.e. in } \Omega\},$$
$$\Lambda_\lambda:=\{(\pi,\Pi)\in X\ |\;\; a((\pi,\Pi),q)=\lambda L(q)\;\;\forall q\in Y\}.$$
From \eqref{yield}, it follows that $P$ is bounded in $X$, i.e. the assumption (B) of Theorem \ref{theorem_duality1} is satisfied. Thus we have
$$\lambda^*=\zeta^*=\inf_{\substack{q\in Y\\ L(q)=1}}\sup_{(\pi,\Pi)\in P}\ a((\pi,\Pi),q)=\inf_{\substack{q\in Y\\ L(q)=1}}\mathcal J(q).$$
In this case, the functional $\mathcal J$ can be found in the form
$$\mathcal J(q)=\int_\Omega \gamma\sqrt{|q|^2+\ell^2|\nabla q|^2}\,dx \quad\forall q\in Y.$$
Although $\mathcal J$ is finite-valued everywhere, it is not coercive in $Y$. Therefore, a certain relaxation of the problem is necessary if we wish to properly define a minimizer of $\mathcal J$ and guarantee its existence. Such an analysis has not been done for this problem and we leave this as a topic for further investigation.

The primal and dual problems have been solved by regularization (penalization) methods in \cite{RS20}. In particular, the regularized functional $\mathcal J_\alpha$ defined by \eqref{J_alpha} takes the form
$$\mathcal J_\alpha(q):=\int_\Omega D_\alpha(q,\nabla q)\,dx,\quad D_\alpha(q,\nabla q)=\left\{
\begin{array}{cl}
\frac{\alpha}{2}(|q|^2+\ell^2|\nabla  q|^2), & \sqrt{|q|^2+\ell^2|\nabla q|^2}\leq\frac{1}{\alpha}\\[2mm]
\sqrt{ |q|^2+\ell^2|\nabla q|^2}-\frac{1}{2\alpha}, &\sqrt{|q|^2+\ell^2|\nabla q|^2}\geq\frac{1}{\alpha}.
\end{array}
\right.$$
Reliable lower and upper bounds of $\lambda^*$ have also been estimated in \cite{RS20} using the regularization methods.

\begin{remark}
\emph{Other choices of yield functions are possible in \eqref{yield}. For example, the following more general function has been considered in \cite{RS20, Reddy2011a}:
\begin{equation}
\varphi_{\ell,r} (\pi,\Pi) : =\left\{
\begin{array}{cc}
\left[|\pi|^{r}+(\ell^{-1}|\Pi|)^{r}\right]^{1/r} - \gamma,& 1\leq r<+\infty,\\
\max\{ |\pi|,\ \ell^{-1}|\Pi| \} - \gamma,& r=+\infty.
\end{array}
\right. 
\label{f_ell_r}
\end{equation}
The set $P$ corresponding to this function remains bounded and thus the equality $\lambda^*=\zeta^*$ holds. Denoting $r'=(1-1/r)^{-1}$ we find the functional $\mathcal J$ in the following form:
\begin{equation}
\mathcal J(q)=
\left\{
\begin{array}{cc}
\int_\Omega \gamma[|q|^{r'}+\ell^2|\nabla q|^{r'}]^{1/r'}\,dx, & 1\leq r'<+\infty,\\[2mm]
\int_\Omega \gamma\max\{ |q|,\ \ell|\nabla q| \}\,dx,& r'=+\infty.
\end{array}
\right.
\label{J_r}
\end{equation}
}
\end{remark}

\subsection{Limit (load) analysis in strain-gradient plasticity}
\label{subsec_ex3}

Limit analysis in gradient-enhanced plasticity has been studied in \cite{Fleck-Willis2009, Polizzotto2010} for a model in which size-dependence is through the gradient of a scalar function of the plastic strain. Here, we consider the model from \cite{Reddy_etal2008, Reddy2011a, CEMRS17, RS20} where the gradient is applied to the entire plastic strain.

We use the same tensors $\sigma$, $\pi$, $\Pi$ and external forces $F$ and $f$ as in Sections \ref{subsec_ex1} and \ref{subsec_ex2}. Let us note that the pair of boundaries $(\Gamma_F,\Gamma_H)$ defined in Section \ref{subsec_ex2} may differ from $(\Gamma_0,\Gamma_f)$ introduced in Section \ref{subsec_ex1}. The limit analysis problem for the strain gradient plasticity reads: \textit{find the supremum $\lambda^*$ over all $\lambda\geq0$ for which there exist $\sigma$, $\pi$, $\Pi$ such that}
\begin{equation}
\mathrm{div}\,\sigma+\lambda F=0 \;\;\mbox{in}\;\Omega,\quad \sigma\nu=\lambda f\;\;\mbox{on}\;\Gamma_f,
\label{balance5}
\end{equation}
\begin{equation}
\sigma^D = \pi - \mbox{div}\,\Pi\quad\mbox{in }\Omega,\quad \Pi\nu= 0\;\;\mbox{on } \Gamma_F,
\label{microforce_balance2}
\end{equation}
\begin{equation}
\varphi_\ell(\pi,\Pi)=\sqrt{|\pi|^2+\ell^{-2}|\Pi|^2}\leq \gamma\quad\mbox{in }\Omega,\quad \gamma,\ell>0.
\label{yield2}
\end{equation}
We see that \eqref{balance5} coincides with \eqref{balance1a} and \eqref{balance1b} from Section \ref{subsec_ex1}. However, we now use the definition of plastically admissible stresses from Section \ref{subsec_ex2} (see \eqref{microforce_balance2} and \eqref{yield2}) instead of \eqref{yield}.

To rewrite this problem in the form \eqref{primal_problem} or \eqref{dual_problem}, we split $\sigma$ as follows:
\begin{equation}
\sigma=pI+\sigma^D=pI+\pi-\mathrm{div}\,\Pi\quad\mbox{in } \Omega.
\label{sigma_split}
\end{equation}
We denote by $X$ the $L^2$-space of all admissible triples $(p,\pi,\Pi)$. The equations \eqref{balance5} and \eqref{microforce_balance2} can be rewritten using \eqref{sigma_split} to the following weak form:
$$a((p,\pi,\Pi),v)=\lambda L(v)\quad\forall v\in Y,$$
where
$$a((p,\pi,\Pi),v):=\int_\Omega [\,p\,\mathrm{div}\,v+\pi:\varepsilon(v)+\Pi\circ\nabla\varepsilon(v)]\,dx,$$
$$L(v):=\int_\Omega F\cdot v\,dx+\int_{\Gamma_f} f\cdot v\,ds,$$
and
$$Y:=\{v\in W^{2,2}(\Omega;\mathbb R^d)\ |\;\; v=0\;\mbox{on } \Gamma_0,\;\; \varepsilon(v)=0\;\mbox{on } \Gamma_H\}.$$
The space $Y$ is equipped with the standard norm denoted by $\|.\|_Y$.
The set $\Lambda_\lambda$ remains the same as in \eqref{Lambda} and the set $P$ of plastically admissible stresses reads 
$$P:=\{(p,\pi,\Pi)\in X\ |\;\; \sqrt{|\pi|^2+\ell^{-2}|\Pi|^2}\leq \gamma\;\;\mbox{a.e. in }\Omega \}.$$ 
Thus, we can define the limit analysis problem as follows:
$$\lambda^*=\sup\{\lambda\in\mathbb R_+\ |\;\; P\cap \Lambda_\lambda\neq\emptyset\}=\sup_{(p,\pi,\Pi)\in P}\inf_{\substack{v\in Y\\ L(v)=1}}\ a((p,\pi,\Pi),v).$$

For analysis of the primal problem \eqref{dual_problem}, it is convenient to use the split $P=P_A+P_C$, where
$$P_A:=\{(p,\pi,\Pi)\in X\ |\;\; p=0,\;\;\sqrt{|\pi|^2+\ell^{-2}|\Pi|^2}\leq \gamma\;\;\mbox{a.e. in }\Omega \},$$
$$P_C:=\{(p,\pi,\Pi)\in X\ |\;\; \pi=0,\;\Pi=0 \}.$$
It is easy to check that $P_A$ is bounded in $X$ and $P_C$ is a closed linear subspace of $X$. We have
$$\zeta^*=\inf_{\substack{v\in Y\\ L(v)=1}}\ \sup_{(p,\pi,\Pi)\in P} a((p,\pi,\Pi),v)=\inf_{\substack{v\in Y\\ L(v)=1}}\mathcal J(v),$$
where
\begin{align*}
\mathcal J(v)&=\left\{\begin{array}{cc}
\int_\Omega \gamma\sqrt{|\varepsilon(v)|^2+\ell^2|\nabla \varepsilon(v)|^2}\,dx, & \mbox{if}\;\;\mathrm{div}\,v=0\;\mbox{in }\Omega,\\
+\infty, & \mbox{otherwise}.
\end{array}\right.
\end{align*}

The inf-sup term in \eqref{inf-sup_abstract} becomes 
\begin{equation}
\inf_{\substack{(p,\pi,\Pi)\in P_{C}\\ (p,\pi,\Pi)\neq0}}\ \sup_{\substack{v\in Y\\ v\neq0}}\ \frac{a((p,\pi,\Pi),v)}{\|(p,\pi,\Pi)\|_X\|v\|_Y} =\inf_{\substack{p\in L^2(\Omega)\\ p\neq0}}\ \sup_{\substack{v\in Y\\ v\neq0}}\ \frac{\int_\Omega p(\mathrm{div}\,v)\,dx}{\|p\|_2\|v\|_Y}.
\label{inf-sup1}
\end{equation}
For the equality $\lambda^*=\zeta^*$ to be satisfied it suffices to show that the right-hand side of \eqref{inf-sup1} is positive on an appropriate factor space of $L^2(\Omega)$. Such an analysis seems to be more involved and we leave this as a topic for further investigation.

\begin{remark}
		\emph{If we replace the yield functions $\varphi_\ell$ in \eqref{yield2} with $\varphi_{\ell,r}$ defined by \eqref{f_ell_r} then the set $P_C$ and the inf-sup expression \eqref{inf-sup1} remain the same. The corresponding functional $\mathcal J(v)$ is the same as in \eqref{J_r} for $\mathrm{div}\,v=0$.
		}
\end{remark}

\subsection{Limit analysis for a delamination problem}

The last example is devoted to a model for delamination, inspired by \cite{Baniotopoulos_Haslinger_2005}. Let $\Omega \subset \mathbb R^2$ denote the domain occupied by an elastic body, with boundary $\partial \Omega$. The body is a laminated composite, comprising two distinct materials. The geometry is idealized with one material, referred to as the bulk, comprising the entire domain with the exception of a thin layer of the second material. This thin layer is treated as a line $\Gamma_b \subset \Omega$, and separation or delamination may occur along this line. 

%Furthermore, interpenetration along $\Gamma$ is excluded by setting the additional condition
%\begin{equation}
%[u] = 0\quad\mbox{on}\ \Gamma\,.
%\label{no_interpen}
%\end{equation}
We follow \cite{Baniotopoulos_Haslinger_2005} and consider a problem with a symmetric geometry and loading, as shown in Figure \ref{domain+loading}(a). Zero displacements in the normal ($x_1$) direction are prescribed along the boundary $\Gamma_\ell$, while on $\Gamma_f$ a surface force $\lambda f$ is applied, where $\lambda\geq0$ is a load factor. The remainder of the boundary $\Gamma_t$ is unconstrained and traction-free. The surface force as well as a body force $\lambda F$ act symmetrically along the $x_1$ axis so that $F(x_1,x_2) = F(x_1,-x_2)$, the same applying to $f$.
\begin{figure}[!h]
\centering
\includegraphics[width=0.45\textwidth]{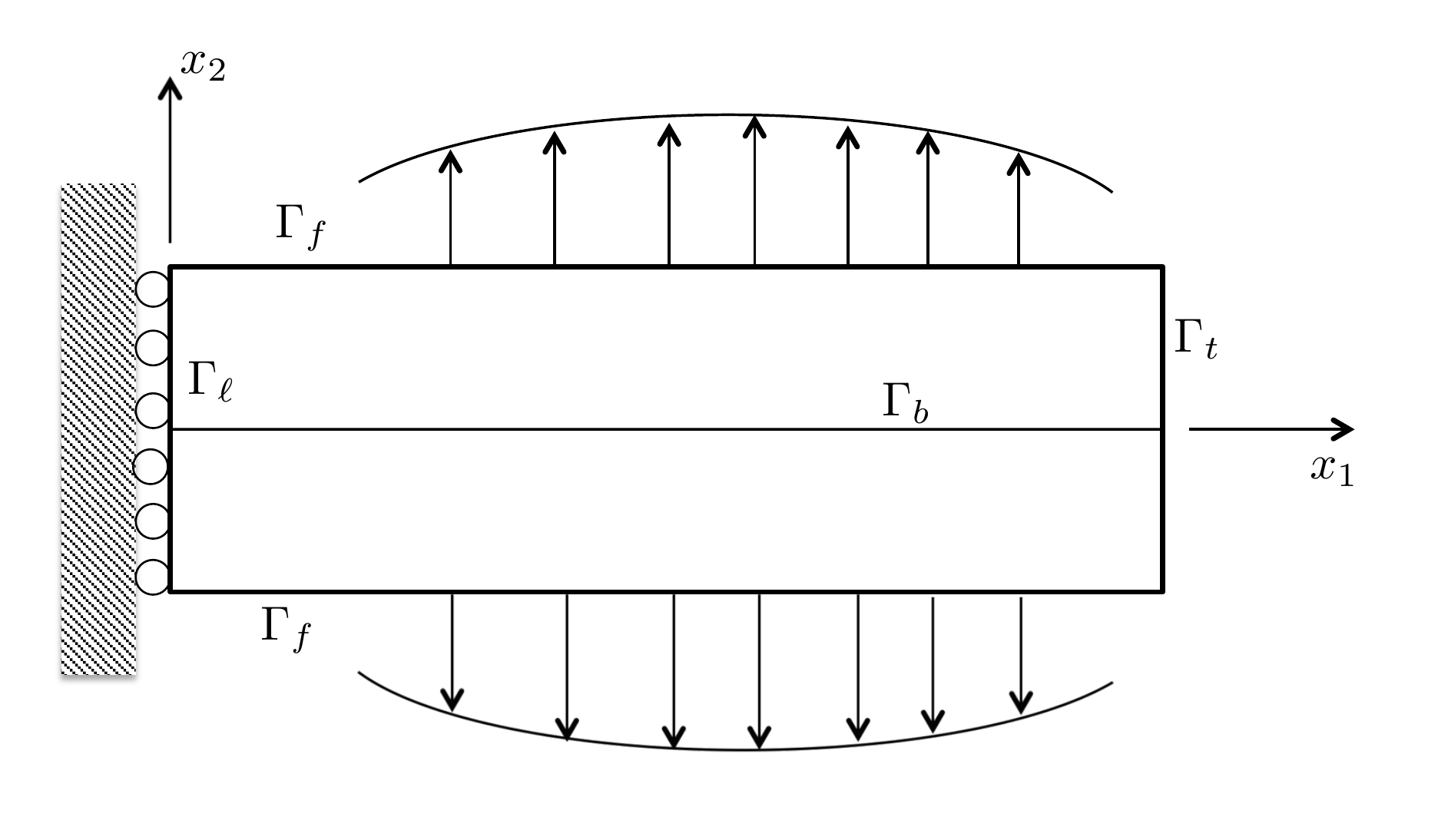} \hspace{6ex}
\includegraphics[width=0.45\textwidth]{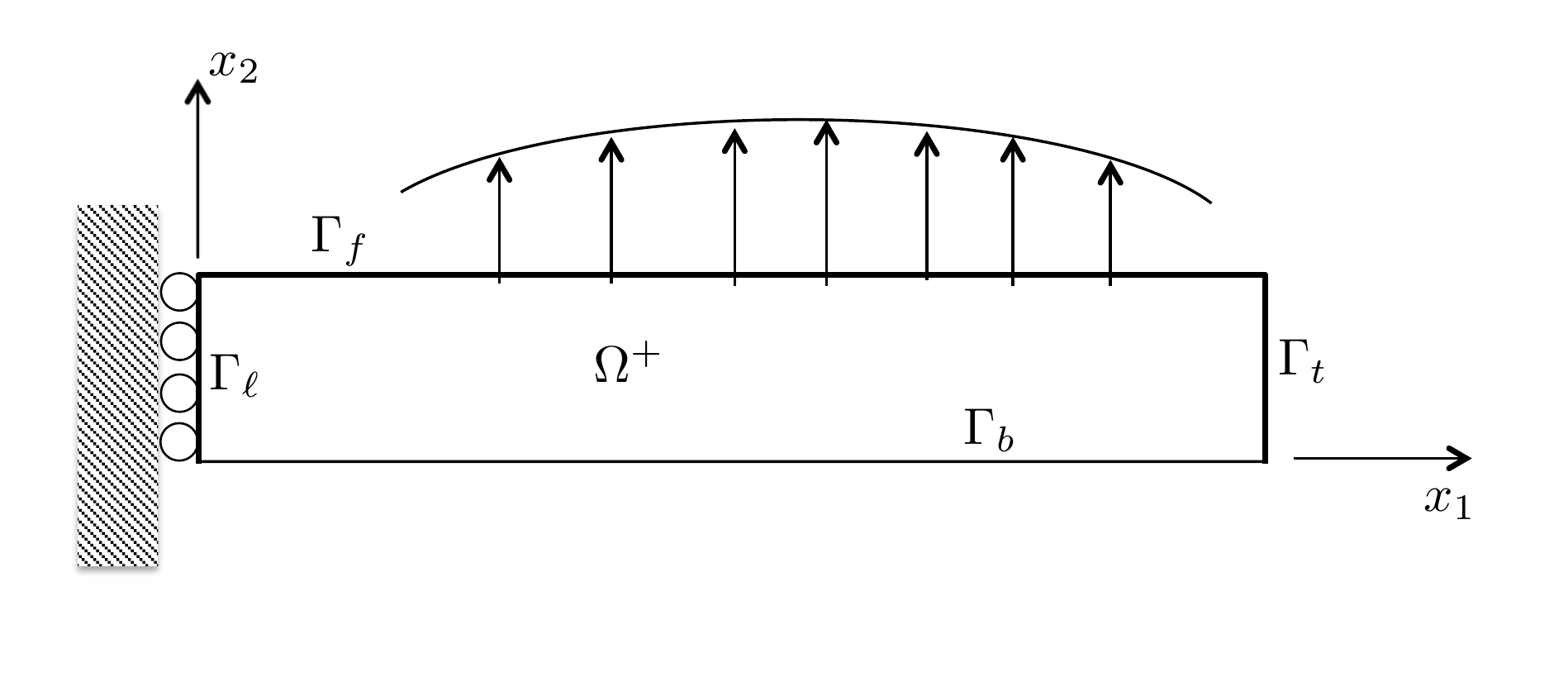}\vspace{3ex}\\
(a) \hspace{42ex} (b)
\caption{(a) Composite body showing domain and loading; (b) Upper half of symmetric body and loading}
\label{domain+loading}
\end{figure}
Given the symmetry of the problem we may confine attention to the upper half $\Omega^+$ of the domain, shown in Figure \ref{domain+loading}(b). 

The boundary conditions set out above have to be augmented with a condition along $\Gamma_b$. This takes the form of conditions on the traction vector $t = \sigma \nu$: from symmetry the tangential component $\sigma\nu\cdot\tau := \sigma_{12}$ must be zero. Here and henceforth subscripts $\nu$ and $\tau$ refer respectively to normal and tangential components. The condition in the normal direction is a constitutive relation that (in the original domain) gives the normal traction $\sigma\nu\cdot\nu := \sigma_{22}$ as a function of the separation $[u_2]$ between the upper and lower surfaces along $\Gamma_b$. Here $u_2$ is the displacement in the normal direction and $[u_2] = u^+_2 - u^-_2$ denotes the jump in displacement at the interface. For the symmetrized problem one may replace the jump $[u_2]$ by $2u_2^+ := 2u_2$.
This has to be supplemented by a non-interpenetration condition, which we do not impose for now, but return to later.

The boundary conditions on $\partial\Omega^+$ are then as follows:
\begin{equation}
\begin{array}{ll}
u_1 = 0,\; \sigma_{12}=0 & \quad\mbox{on}\ \Gamma_\ell, \\
\sigma\nu = \lambda f & \quad\mbox{on}\ \Gamma_f, \\
\sigma\nu = 0 & \quad\mbox{on}\ \Gamma_t, \\
\sigma_{12} = 0,\ \ 
\sigma_{22} (x_1) \in H(u_2(x_1)) & \quad \mbox{on}\ \Gamma_b,
\end{array}
\label{bcs}
\end{equation}
where $H$ denotes a multivalued step function in $\mathbb R^1$. Examples of $H$ can be found in \cite{Baniotopoulos_Haslinger_2005}. For purposes of this paper, we shall assume that the values of $H$ belong to the interval $[-\gamma,\gamma]$ where $\gamma > 0$ is a prescribed threshold for delamination. Then $H$ can be either the projection of $\mathbb R^1$ onto $[-\gamma,\gamma]$ or the multifunction
$H (x) = \gamma\,\mathrm{sign}\,x$ for $x\neq 0$ and $H(0) \in [-\gamma,\gamma]$.

The bulk material is modelled as linear elastic, to which we add the equilibrium equation on $\Omega^+$:
\begin{align}
\mathrm{div}\,\sigma + \lambda F = 0. 
\label{equil}
\end{align}
%On $\Gamma_b$, the body is glued in $x_2$-direction and free in $x_1$-direction. It means (in our case) that the stress component $\sigma_{22}$ is bounded, that is $|\sigma_{22}|\leq \gamma$ on $\Gamma_b$ where $\gamma>0$ is a given value. Unlike \cite{Baniotopoulos_Haslinger_2005}, we do not consider any non-penetration condition there. On $\Gamma_l$, we assume that the body cannot move in the normal ($x_1$) direction. On $\Gamma_t$, a traction $f$ multiplied by a load factor $\lambda$ is prescribed and the boundary $\Gamma_r$ is free. The volume force $\lambda F$ is prescribed in $\Omega$.

%\begin{figure}[htbp]
%	\centering
%	\begin{picture}(230,100)
%	
%	{\thicklines
%		\put(20,80){\line(1,0){200}}
%		\put(20,20){\line(0,1){60}}
%		\put(220,20){\line(0,1){60}}
%	}
%	\multiput(15,22)(0,5){12}{\circle{4}}
%	\multiput(120,82)(6,0){17}{\vector(0,1){10}}
%	\put(30,30){\vector(1,0){20}}
%	\put(30,30){\vector(0,1){20}}
%	
%	\put(170,94){\makebox(0,0)[b]{$\lambda f$}}
%	\put(50,33){\makebox(0,0)[b]{$x_1$}}
%	\put(33,50){\makebox(0,0)[l]{$x_2$}}
%	\put(120,17){\makebox(0,0)[t]{$\Gamma_b$}}
%	\put(10,50){\makebox(0,0)[r]{$\Gamma_l$}}	
%	\put(120,77){\makebox(0,0)[t]{$\Gamma_t$}}
%	\put(223,50){\makebox(0,0)[l]{$\Gamma_r$}}
%	\put(120,45){\makebox(0,0)[c]{$\Omega$}}			
%	
%	{\linethickness{2pt} \put(20,20){\line(1,0){200}}}	
%	\end{picture}
%	\caption{Geometry of the strip-footing problem.}
%	\label{fig_geometry}
%\end{figure}

The limit load for the problem can be defined formally as follows: \textit{find the supremum $\lambda^*\geq0$ over all $\lambda\geq 0$ for which there exists a stress field $\sigma\colon\Omega^+\rightarrow \mathbb R^{2\times 2}_{sym}$ that satisfies \eqref{equil} and}
\begin{equation}
\begin{array}{ll}
\sigma_{12}=0 & \quad\mbox{on}\ \Gamma_\ell, \\
\sigma\nu = \lambda f & \quad\mbox{on}\ \Gamma_f, \\
\sigma\nu = 0 & \quad\mbox{on}\ \Gamma_t, \\
\sigma_{12} = 0,\ \ 
|\sigma_{22}|\leq\gamma & \quad \mbox{on}\ \Gamma_b.
\end{array}
\label{bcs2}
\end{equation}

To rewrite this problem in the form \eqref{primal_problem} and \eqref{dual_problem}, we introduce an auxiliary variable $\Xi\in L^2(\Gamma_b)$ that coincides with $-\sigma_{22}$ on $\Gamma_b$ in a weak sense. Then the space $X=L^2(\Omega^+,\mathbb R^{2\times 2}_{sym})\times L^2(\Gamma_b)$ contains pairs $(\sigma,\Xi)$ and  
$$Y:=\{v=(v_1,v_2)\in W^{1,2}(\Omega^+,\mathbb R^2)\ |\;\; v_1=0\;\mbox{on } \Gamma_\ell\}$$
consists of admissible displacement fields. Using the spaces $X$ and $Y$ one can rewrite the equations in \eqref{equil}--\eqref{bcs2} in the following weak form:
$$a((\sigma,\Xi),v)=\lambda L(v) \quad\forall v\in Y,$$
where
$$a((\sigma,\Xi),v)=\int_{\Omega^+} \sigma:\varepsilon(v)\,dx+\int_{\Gamma_b}\Xi v_2\,dx, \quad \varepsilon(v)=\frac{1}{2}[\nabla v+(\nabla v)^\top]$$
and
$$L(v)=\int_{\Omega^+} F\cdot v\,dx+\int_{\Gamma_f}f\cdot v\,ds,\quad v\in Y.$$
The set $P$ and its decomposition into $P_A$ and $P_C$ are defined as follows:
$$P:=\{(\sigma,\Xi)\in X\ |\;\; |\Xi|\leq\gamma\;\mbox{in } \Gamma_b\},\quad P_C:=\{(\sigma,\Xi)\in X\ |\;\; \Xi=0\;\mbox{on } \Gamma_b\},$$
$$P_A:=\{(\sigma,\Xi)\in X\ |\;\; \sigma=0\;\mbox{on }\Omega^+,\; |\Xi|\leq\gamma\;\mbox{on } \Gamma_b\}.$$ We also define $\Lambda_\lambda := \{ (\sigma,\Xi) \in X\ |\ a((\sigma,\Xi),v) = \lambda L(v)\ \ \forall v \in Y\}$. Then, the dual and primal problems read
\begin{equation}
\lambda^*=\sup\{\lambda\in\mathbb R_+\ |\;\; P\cap \Lambda_\lambda\neq\emptyset\}=\sup_{(\sigma,\Xi)\in P}\inf_{\substack{v\in Y\\ L(v)=1}}\ a((\sigma,\Xi),v)
\label{dual_delam}
\end{equation}
and 
\begin{equation}
\zeta^*=\inf_{\substack{v\in Y\\ L(v)=1}}\sup_{(\sigma,\Xi)\in P} a((\sigma,\Xi),v)=\inf_{\substack{v\in Y\\ L(v)=1}}\mathcal J(v).
\label{primal_delam}
\end{equation}

To show that $\lambda^*=\zeta^*$ we use Theorem \ref{theorem_duality6}. In particular, we have
$$H=\{(\sigma,\Xi)\in X\ |\;\; a((\sigma,\Xi),v)=0\;\;\forall v\in Y\}$$
and
$$P_C/H=\{(\sigma,\Xi)\in X\ |\;\;\exists v\in Y:\; \sigma=\varepsilon(v),\;\;\Xi=0\}.$$
The latter identity follows, for example, from \cite{NH17}. Then the inf-sup condition \eqref{inf-sup_abstract4} is a consequence of the Korn inequality \cite{NH17}. 

In addition, if $\lambda^*=\zeta^*<+\infty$ then one can find analytical solutions $v^*\in Y$ and $(\sigma^*,\Xi^*)\in P\cap\Lambda_{\lambda^*}$ to \eqref{primal_delam} and \eqref{dual_delam}, respectively. Indeed, from \eqref{J0}, \eqref{K} and Lemma \ref{lem_J}, it follows that
$$\mathcal J(v)=\left\{
\begin{array}{cl}
\int_{\Gamma_b}\gamma|v_2|\,dx,& v\in\mathcal K,\\
+\infty,& v\not\in\mathcal K,
\end{array}
\right.,\quad \mathcal K=\{v\in Y\ |\;\; v=(0,q),\; q\in\mathbb R\},$$
that is, $\mathrm{dom}\,\mathcal J=\mathcal K$. It is readily seen that the feasible set $\mathrm{dom}\,\mathcal J\cap \{v\in Y\ |\; L(v)=1\}$ in \eqref{primal_delam} is the singleton consisting of the function
$$v^*=(v_1^*, v_2^*),\quad v_1^*=0,\quad v_2^*=\left(\int_{\Omega^+} F_2\,dx+\int_{\Gamma_f}f_2\,ds\right)^{-1},$$
provided that $\int_{\Omega^+} F_2\,dx+\int_{\Gamma_f}f_2\,ds\neq0$. If it is so then $v^*$ is also the unique solution to the primal problem \eqref{primal_delam} and 
$$\lambda^*=\zeta^*=\gamma|\Gamma_b|\Big|\int_{\Omega^+} F_2\,dx+\int_{\Gamma_f}f_2\,ds\Big|^{-1}<+\infty.$$
By analysis of the saddle-point problem related to \eqref{primal_delam} and \eqref{dual_delam}, we find that
the solution $(\sigma^*,\Xi^*)$ to the dual problem \eqref{dual_delam} satisfies $\Xi^*=\gamma\mathrm{sign}(v_2^*)$ and 
\begin{equation}
\int_{\Omega^+} \sigma^*:\varepsilon(v)\,dx=\lambda^*L(v)-\int_{\Gamma_b}\Xi^*v_2\,ds\quad\forall v\in Y.
\label{sigma*}
\end{equation}
The component $\sigma^*$ is not uniquely defined. One of $\sigma^*$ satisfying \eqref{sigma*} is the elastic stress of the form $\sigma^*=\mathbb C\varepsilon(u^*)$ in $\Omega^+$, where $u^*\in Y$ and $\mathbb C$ is the elastic fourth order tensor representing Hooke's law. If $\int_{\Omega^+} F_2\,dx+\int_{\Gamma_f}f_2\,ds=0$ then $\lambda^*=\zeta^*=+\infty$.

\begin{remark}
	\emph{If we consider the case in which the body is fixed on $\Gamma_\ell$ as in \cite{Baniotopoulos_Haslinger_2005}, then $\mathcal K=\{0_Y\}$, which implies that $\lambda^*=\zeta^*=+\infty$. Thus the related delamination problem may have a solution even if the composite is completely debonded.}
\end{remark}

\begin{remark}
\emph{The complete formulation of the delamination problem requires also a condition of non-interpenetration (that is, a Signorini condition) along $\Gamma_b$. For the symmetrized problem this amounts to defining the conic set
$Y_C:=\{v\in Y\ |\;\; v_2\geq0 \;\mbox{on } \Gamma_b\}$ of admissible displacement fields, replacing the last of equations \eqref{bcs2} with
$$\sigma_{21}=0,\;\;-\sigma_{22}\in[0,\gamma] \quad \mbox{on}\;\Gamma_b, \label{Gamma_b2}$$
and consequently, replacing $P$ with
$P:=\{(\sigma,\Xi)\in X\ |\;\; \Xi\in[0,\gamma]\;\mbox{on } \Gamma_b\}$. According to Theorem \ref{theorem_contact}, we have the duality problem 
\begin{equation*}
\lambda^*=\sup_{x\in P}\inf_{\substack{y\in Y_C\\ L(y)=1}}\ a(x,y) \stackrel{?}{=}\inf_{\substack{y\in Y_C\\ L(y)=1}}\sup_{x\in P}\ a(x,y)=\zeta^*.
\label{duality_problem3}
\end{equation*}
By combining Theorems \ref{theorem_duality6} and \ref{theorem_contact} it is possible to show that $\lambda^*=\zeta^*$. In particular, if 
$$\int_{\Omega^+} F_2\,dx+\int_{\Gamma_f}f_2\,ds>0$$
we obtain the same limit value and the primal and dual solutions as for the duality problem without the non-penetration condition.}
\end{remark}

\section{Conclusion}

This work has been concerned with an inf-sup problem posed on abstract Banach spaces. The main feature of this convex and constrained problem has been the presence of a bilinear Lagrangian, which appears in applications leading to linear, cone or convex programming problems. Conditions for ensuring duality without any gap have been introduced. We have introduced and extended an innovative framework based on an inf-sup condition on convex cones generalizing the well-known Babu\v ska-Brezzi conditions. We have also suggested a new regularization method and derived a computable majorant to the problem.

Applications of the abstract problem to various examples in mechanics have been presented. First, the problem of limit analysis in classical plasticity has been revisited in the context of the duality framework of this work. Then, we have shown that the abstract framework may be used in the case of two different subproblems related to strain-gradient plasticity, viz. the determination of plastically admissible stresses and the determination of limit loads, and for a delamination problem.

The techniques presented in this paper could be extended to more general duality problems where the Lagrangian contains, in addition to the bilinear form, linear forms with respect to primal or dual variables. Such an extension would be applicable to a wider range of problems in mechanics.

{\bf Acknowledgment:} SS and JH acknowledge support for their work from the Czech Science Foundation (GA\v{C}R) through project No. 19-11441S. BDR acknowledges support for his work from the National Research Foundation, through the South African Chair in Computational Mechanics, SARChI Grant 47584.

%%%%%%%%%%%%% references %%%%%%%%%%%%%%%%%%%%%%%%%%%%%%%

\end{document}